\documentclass{c0PFH}

\begin{document}

\title{PFH spectral invariants on the two-sphere and the large scale geometry of Hofer's metric.}

\date{\today}
\author{Daniel Cristofaro-Gardiner, Vincent Humili\`ere, Sobhan Seyfaddini }
\maketitle

\begin{abstract} 
We resolve three longstanding questions 
related to the large scale geometry of the group of Hamiltonian diffeomorphisms of the two-sphere, equipped with Hofer's metric.
Namely: (1) we resolve the Kapovich-Polterovich question by showing that this group is not quasi-isometric to the real line; (2) more generally, we show that the kernel of Calabi over any proper open subset is unbounded; and (3) we show that the group of area and orientation preserving homeomorphisms of the two-sphere is not a simple group.   We also obtain, as a corollary, that the group of area-preserving diffeomorphisms of the open disc,  equipped with an area-form of finite area, is not perfect.  Central to all of our proofs are new sequences of spectral invariants over the two-sphere, defined via periodic Floer homology.
\end{abstract}

\tableofcontents

\bigskip

\section{Introduction}
It is a remarkable fact that  the group of Hamiltonian diffeomorphisms of a symplectic manifold admits a bi-invariant Finsler metric, known as {\bf Hofer's metric}.  The existence of such a metric on an infinite dimensional Lie group is highly unusual, due to the lack of compactness, and stands in contrast to the fact that a simple\footnote{$\Ham(M, \omega)$ is simple for closed $M$, by a theorem of Banyaga
 \cite{Banyaga}.} finite dimensional Lie group admits a bi-invariant Finsler metric only if it is compact; see \cite[Prop.\ 1.3.15]{Polterovich-Rosen}.

 \sloppy The theme of this article is the large-scale geometry of Hofer's metric, on $\Ham(\S^2, \omega)$, the Hamiltonian diffeomorphisms of the $2$--sphere\footnote{It is known that the group $\Ham(\S^2, \omega)$  is in fact the set of all diffeomorphisms of $\S^2$ which preserve the area form $\omega$.}.  Our 
 first result, Theorem \ref{theo:QI-ker-Cal}, settles two longstanding questions, presented below, about the quasi-isometry type of $\Ham(\S^2, \omega)$. 
The first of the two questions was posed by Kapovich and Polterovich in 2006.
 
\begin{question}\label{que:Kap-Pol}
Is $\Ham(\S^2, \omega)$ quasi-isometric to the real line $\R$? 
\end{question}

The second question is due to Polterovich and dates back to the 2000s.
To state it, consider a connected, proper open set $U \subset \S^2$ and  denote by $\Ham_U(\S^2, \omega)$ the subgroup of $\Ham(\S^2,\omega)$ consisting of Hamiltonian diffeomorphisms supported in $U$.  This subgroup carries a well-known group homomorphism called the \textbf{Calabi homomorphism}:
 \[\Cal:\Ham_U(\S^2,\omega)\to \R,\]
 whose definition we recall in Section \ref{sec:prelim_symp}; see Equation \eqref{eqn:def_calabi}. 
 
 \begin{question}\label{que:Ker_Cal}
 Suppose\footnote{ If $\area(U) > \frac 12 \area(\S^2)$, then the question is known to have an affirmative answer by Polterovich \cite{Polterovich98}.} that $\area(U) \le \frac 12 \area(\S^2)$. Is the kernel of $\Cal:\Ham_U(\S^2,\omega)\to \R$ an unbounded subset of $\Ham(\S^2, \omega)$? 
 \end{question}

    The Hofer geometry of the two-sphere has long remained mysterious, and these two basic questions have received much attention over the past years.  This is especially the case for Question \ref{que:Kap-Pol}, which appears as Problem 21 on the list of open problems of McDuff-Salamon \cite[Sec. 14.2]{McDuff-Salamon};  it is mentioned as one of the motivations behind the influential article of Polterovich and Shelukhin \cite[Sec.\ 1.3]{Pol-Shel}; and it is highlighted in several articles such as  \cite{Py, EPP, Kislev-Shelukhin, Brandenbursky-Shelukhin}. 

    We also continue the direction of research initiated in our recent article \cite{CGHS}.  In particular, we answer the following question from the 1980 article of Fathi \cite{fathi} on the algebraic structure of $\Homeo_0(\S^2, \omega)$, the group of all area and orientation preserving homeomorphisms of the 2-sphere\footnote{The group $\Homeo_0(\S^2, \omega)$ can alternatively be described as the connected component of the group of area preserving homeomorphisms of $\S^2$. For any transformation group, the simplicity question is only interesting for the component of the identity because it forms a normal subgroup of the larger group.}.  
\begin{question}\label{que:sphere}
Is the group $\Homeo_0(\S^2, \omega)$ simple? 
\end{question} 
 Although at first glance this question might appear unrelated to Hofer's geometry, we will see that the large scale geometry of Hofer's metric plays a crucial role in the solution.  The two-sphere is the only closed manifold for which the question of simplicity of the component of the identity in the group of volume-preserving homeomorphisms remained open; for other closed manifolds this was settled by Fathi in the late 1970s.

 \subsection{The large-scale geometry of the kernel of Calabi}
\label{sec:large}

Let $d_H$ denote the Hofer metric on  $\Ham(M, \omega)$, the group of Hamiltonian diffeomorphisms of a closed and connected symplectic manifold $(M, \omega)$; we will review the definition of $d_H$, and other basic notions from symplectic geometry, in Section \ref{sec:prelim_symp}.  

A fundamental notion in large-scale geometry is that of {\bf quasi-isometry},   which we now recall.  A {\bf quasi-isometric embedding} is a mapping $\Phi :(X_1, d_1)  \rightarrow (X_2, d_2)$ of metric spaces for which there exist constants $A \ge 1,  B \ge 0$ such that 
\begin{equation}\label{def:quasi_isom}
\frac{1}{A}d_1(x, y) -B \leq d_2(\Phi(x), \Phi(y)) \leq A \,d_1(x, y) + B.
\end{equation} 
The map $\Phi$, satisfying the above, is said to be a quasi-isometry if it is {\bf quasi-surjective}, i.e.\ if  there exists a constant $C>0$ such that every point in $X_2$ is within distance $C$ of the image $\Phi(X_1)$. 

The large-scale geometry  of Hofer's metric, on general symplectic manifolds, has been studied extensively ever since Hofer's discovery of the metric in 1990 \cite{Hofer-metric}; see for example  \cite{Ostrover, EP03, Usher13,  Py, Khanevsky09, Humiliere, Sey14, Khanevsky16, Pol-Shel, Alvarez-Gavela}.  
Usually, $(\Ham, d_H)$ is a ``large" metric space.   For example, it is conjectured to be always  unbounded, and this has been proven for many manifolds \cite{Lalonde-McDuff2, Polterovich98, Schwarz, Ostrover, EP03, McDuff10, Usher13}.  Moreover, Usher \cite{Usher13} has proven that, for a large class of manifolds, including closed surfaces of positive genus,\footnote{As observed in \cite{Py}, such results for surfaces of positive genus can be deduced from the arguments in \cite{Lalonde-McDuff2,  Polterovich98}.}  it admits a quasi-isometric embedding of infinite-dimensional normed vector spaces; see also Py's article \cite{Py}.    
 
Despite all the above progress, a famous case that has been difficult to understand is that of the two-sphere.  All that is known is that $\Ham(\S^2, \omega)$, and the subgroup $\Ham_U(\S^2,\omega)$, are unbounded and admit a quasi-isometric embedding of the real line $\R$; this was proven by Polterovich  \cite{Polterovich98}.  As for the kernel of $\Cal:\Ham_U(\S^2,\omega)\to \R$, with $\area(U) \le \frac 12 \area(\S^2)$, it is not even known if it is unbounded, i.e.\ whether it is quasi-isometric to the point.    It is our understanding that when Question~\ref{que:Kap-Pol} and Question~\ref{que:Ker_Cal} were posed, there were not even clear conjectures about what their answers should be. 

Our first point in the present work is that the kernel of Calabi is indeed rather big, which we illustrate in two different ways. 

\begin{theo}\label{theo:QI-ker-Cal}   Let $U\subset \S^2$, with $U\neq \S^2$.  Then:
\begin{enumerate}[(a)]
\item For any $n\in\N$, there exists a quasi-isometric embedding of $\R^n$ into $(\Ham(\S^2, \omega), d_H)$ whose image is contained in 
 the kernel of the Calabi homomorphism $\Cal:\Ham_U(\S^2,\omega)\to \R$. 
\item  The kernel of  $\Cal:\Ham_U(\S^2,\omega)\to \R$ is not coarsely proper. 
\end{enumerate}
\end{theo}

To review the terminology here, recall that a metric space $(X,d)$
is said to be {\bf coarsely proper} if there exists $R_0 >0$ such that every bounded subset of $(X,d)$ can be covered by finitely many balls of radius $R_0$; see \cite[Definition 3.D.10]{Cornulier-delaHarpe}.  Examples of coarsely proper spaces include the Euclidean space $\R^n$ or any bounded spaces ---  in particular, part (b) of Theorem~\ref{theo:QI-ker-Cal} resolves Question~\ref{que:Kap-Pol} and Question~\ref{que:Ker_Cal} --- but on the other hand, an infinite-dimensional Banach space is not coarsely proper.    Recall also that a {\bf quasi-flat} in a metric space $(X,d)$ is the image of a quasi-isometric embedding of $\R^n$; moreover,  the {\bf quasi-flat rank} of a metric space $(X,d)$ is the supremum, over all $n$, such that there exists a quasi-isometric embedding of $\R^n$ into $X$. Thus, part (a) of Theorem~\ref{theo:QI-ker-Cal} is equivalent to the statement that the metric space $(\Ham(\S^2, \omega), d_H)$ and  the subset given by the kernel of the Calabi homomorphism $\Cal:\Ham_U(\S^2,\omega)\to \R$  have infinite quasi-flat rank.   Now,  it is known that the quasi-flat rank of $\R^n$ is $n$ and so we see that part (a) of Theorem~\ref{theo:QI-ker-Cal} also answers Questions \ref{que:Kap-Pol} and \ref{que:Ker_Cal}.  In fact, we will see in Example~\ref{ex:main} below that Theorem \ref{theo:QI-ker-Cal} tells us quite a bit more about the quasi-isometry type of the metric spaces in question.

\begin{example}
\label{ex:main}
  Let $(G,d)$ be a finite dimensional connected Lie group, with a left invariant Finsler metric induced from a norm on its Lie algebra; we call such a $d$ a {\bf compatible metric.}  As was explained above, the existence of Hofer's metric dramatically contrasts the situation for finite dimensional Lie groups; one might hope that the large-scale geometry also sees this.  Indeed it is known that any such $(G,d)$ both has finite quasi-flat rank, and is coarsely proper.  So, our main theorem precludes this as a quasi-isometry type for $(\Ham(\S^2), d)$ or for the kernel of Calabi.  Similarly,  any finitely generated group, or more generally, any locally compact and compactly generated group (here we refer the reader to \cite{Cornulier-delaHarpe} for the precise definition) is coarsely proper, see  \cite[Proposition 3.D.29]{Cornulier-delaHarpe}.   It would be interesting to understand to what degree the quasi-isometry type of $\Ham(\S^2)$ is unique, for example whether it differs\footnote{We have learned in recent conversation with Polterovich that this question is wide open.} from that of $\Ham(S)$ for other surfaces $S$.    
\end{example}

\begin{remark}
\label{rmk:ps}
Contemporaneously with our work, Polterovich-Shelukhin have shown \cite{Pol-Shel2}, using very different methods, that there is an isometric embedding of the space of even compactly supported functions on $(-\frac18,\frac18)$ into $\Ham(\S^2,\omega).$  This clearly answers the Kapovich-Polterovich question and, moreover, implies that $\Ham(\S^2, \omega)$ is neither coarsely proper nor of finite quasi-flat rank.  It would be very interesting to relate our methods here to the methods in \cite{Pol-Shel2}.  
\end{remark}

\subsection{Non-simplicity of $\Homeo_0(\S^2, \omega)$}
   We turn now to continuous symplectic geometry.
   
   In our recent article \cite{CGHS}, we proved that the group of compactly supported area-preserving homeomorphisms of the disc is not simple.  Our next theorem settles the simplicity question for the sphere.  Recall that $\Homeo_0(\S^2, \omega)$ denotes the identity component in the group of area-preserving homeomorphisms of the two-sphere.

\begin{theo}\label{theo:non-simplicity}
  $\Homeo_0(\S^2, \omega)$ is not simple.
\end{theo}

In fact, as in our previous article \cite{CGHS}, this theorem implies a stronger statement by appealing to a beautiful argument of Epstein and Higman \cite{Epstein, Higman}.  Recall that a group is {\bf perfect} if it is equal to its commutator subgroup.

\begin{corol}\label{cor:non-perfect}
$\Homeo_0(\S^2, \omega)$ is not perfect.
\end{corol}

Theorem~\ref{theo:non-simplicity} answers a question of Fathi\footnote{In fact, Theorem \ref{theo:FHomeo_proper}, stated below, answers Fathi's question for all compact genus zero surfaces; see Remark \ref{rem:genus_zero_surfaces}.} from the 70s \cite[Appendix A.6]{fathi}, whose history we now briefly review.  The question of simplicity of groups of homeomorphisms and diffeomorphisms was studied extensively in the 50s, 60s, and 70s and is fairly well-understood in most scenarios.  However, area-preserving homeomorphisms of surfaces have remained mysterious.  For example, in the case of closed manifolds, the simplicity question  had been answered by the late 70s for all of the following groups: homeomorphisms, diffeomorphisms\footnote{We are considering $C^\infty$ diffeomorphims here. For $C^k$ diffeomorphisms, simplicity is known for all $k$ except when $k = \mathrm{dim(M)}+1$ which remains open to this date.},  volume-preserving diffeomorphisms and symplectomorphisms.  And in the case of volume preserving homeomorphisms it was answered by Fathi \cite{fathi} for every closed manifold other than the two sphere.  Fathi asked the aforementioned question answered by Theorem~\ref{theo:non-simplicity} in the work \cite{fathi}.

We remark that non-simplicity of $\Homeo_0(\S^2, \omega)$ is surprising as it stands in dramatic contrast to the fact that on closed simply connected manifolds, such as spheres of dimension greater than one, this is the only example of the ``usual" transformation groups known to be non-simple.  For example, it is known that for simply connected manifolds the identity component in any of the groups mentioned in the previous paragraph is  simple except, of course, in our case of area-preserving homeomorphisms of the sphere. 

The simplicity of the aforementioned groups was established through the works of a long list of mathematicians who studied  the question from the 30s to the late 70s. For a  
summary of the long history of the simplicity question, we refer the interested reader to \cite[Sec.\ 1]{CGHS}.

\subsubsection{The perfectness question for volume-preserving diffeomorphisms of $\R^n$}


We now explain an application to the study of the algebraic structure of diffeomorphism groups.

Let $\Omega$ be a volume form on $\R^n$ and denote by $\Diff(\R^n,\Omega)$ the group of all diffeomorphisms of $\R^n$ which preserves $\Omega$.    McDuff proved in 1980 that although $\Diff(\R^n,\Omega)$ is non-simple\footnote{In this case, non-simplicity follows from the fact that the  compactly supported volume-preserving diffeomorphisms form a proper normal subgroup.}, it is  always perfect for $n\geq 3$; see \cite{McDuff-perfectness-volume-preserving}.   There are two distinct cases of McDuff's theorem, namely the finite volume case, which is the same as the case of an open ball with its standard volume form, and the infinite volume case.   In both cases, however, the 2-dimensional case has remained open.   
Theorem \ref{theo:non-simplicity} allows us to settle this question in the finite area case.

\begin{corol}\label{non-perfectness-R2} Assume that $\int_{\R^2}\Omega<+\infty$. Then, $\Diff(\R^2,\Omega)$ is not perfect.  
\end{corol}
We prove this corollary in Section \ref{sec:non-perfectness-R2}.

\subsection{New spectral invariants} \label{sec:intro_tools}
 
We now discuss the main tools that we use and develop here for proving the above theorems.  We henceforth view $\S^2$ as the unit sphere in standard $\R^3$ and equip it with the symplectic form  $\omega := \frac{1}{4 \pi} d \theta \wedge dz,$ where $(\theta,z)$ are cylindrical coordinates.  Note that this gives the sphere a total area of $1$.

\subsubsection*{Periodic Floer homology and spectral invariants}

To prove our results we use a version of Floer homology for area-preserving diffeomorphisms called {\bf periodic Floer homology} (PFH) which was introduced by Hutchings  \cite{Hutchings-Sullivan-Dehntwist}; we will review PFH in Section \ref{sec:prelim_PFH}.  As will be reviewed in Section \ref{sec:newspec}, one can use PFH to define a collection of invariants of Hamiltonians on the sphere
$$c_{d,k} : C^\infty(\S^1 \times \S^2) \rightarrow \R$$
which are indexed by $d\in \N$ and $k \in \Z$ with $k$ having the same parity as $d$.  
 
We show in Section \ref{sec:newspec} that these invariants have various useful properties; see Proposition \ref{prop:more}.  In particular, we show that they can be used to define invariants 
\begin{gather*}c_{d,k} :  \widetilde{\Ham}(\S^2, \omega) \rightarrow \R, \\
c_d:  \widetilde{\Ham}(\S^2, \omega) \rightarrow \R,
\end{gather*}   
where $c_d := c_{d,-d}$, which are well-defined on the universal cover of $\Ham(\S^2, \omega)$.  Moreover, we show in Proposition \ref{prop:cdeven}, that if $d$ is even then $c_{d,k} :  \widetilde{\Ham}(\S^2, \omega) \rightarrow \R$ descends to $\Ham(\S^2, \omega)$ and so in particular we obtain
$$c_d :  \Ham(\S^2, \omega) \rightarrow \R,$$
defined for even $d$.  

\subsection*{Homogenization}

As is evident from the works of Entov-Polterovich \cite{EP03, EP09}, for the purposes of applications to Hofer's geometry, it is often beneficial to homogenize spectral invariants.  This is true in our work as well and, in fact, we prove Theorem \ref{theo:QI-ker-Cal} using the homogenizations of the invariants $c_d$ which we now introduce.  More precisely, we can define for $\varphi \in \Ham(\S^2, \omega)$, and for all $ d \in \N$, 

\begin{equation}\label{eq:mu_d}
 \mu_d(\varphi) := \limsup_{n \to \infty} \frac{ c_d(\tilde \varphi^n)}{n},
 \end{equation}
 where $\tilde \varphi \in \widetilde{\Ham}(\S^2, \omega)$ is any lift of $\varphi$; we show in Proposition \ref{prop:mu_d-well-defined} that the above $\limsup$ is well defined and that $\mu_d(\varphi)$ does not depend on the choice of $\tilde \varphi \in \Hamtilde(\S^2, \omega)$.  We also define the related invariant
$\zeta_d : C^{\infty}( \S^2) \rightarrow \R$ by
\begin{equation}\label{eqn:zeta_d_def}
\zeta_d(H):= \limsup_{n \to \infty}\frac{c_{d}(nH)}{n}.
\end{equation}
We will see that these two homogenized invariants are related by the formula 
$$\mu_d(\varphi^1_H) = \zeta_d(H) - d \int_{\S^2} H \, \omega.$$

A useful property of any $\mu_d$ is that it coincides with (a multiple of) the Calabi invariant for Hamiltonian diffeomorphisms with small supports.  More precisely, suppose that $\supp(\varphi)$,  the support of $\varphi \in \Ham(\S^2, \omega)$, is contained in a topological disc $D$ with $\area(D) < \frac{1}{d+1} $.  Then, the following equality holds
\begin{equation} \label{eq:energy_cap}
\frac{1}{d} \mu_d(\varphi) = -\Cal(\varphi).
\end{equation}

The above properties of $\mu_d, \zeta_d$ will be proven in Section \ref{sec:newspec}.

\begin{remark}
The properties of the $\mu_d$ are reminiscent of the {\bf Calabi quasimorphism} of Entov-Polterovich \cite{EP03}.  It is an open question whether $\Ham(\S^2, \omega)$ admits any Hofer continuous (homogeneous) quasimorphisms other than the one constructed by Entov-Polterovich.  We plan to investigate in future work whether the invariants $\mu_d$ are quasi-morphisms. 
\end{remark}

\subsection*{The Hofer Lipschitz property and monotone twists}
A critical fact which we will show, and which is at the heart of all applications to Hofer's geometry is the {\bf Hofer Lipschitz} property.  For the invariants $\mu_d$ this means that  the following holds: 
$$\vert \mu_d(\varphi) - \mu_d(\psi) \vert \leq C_d \, d_H(\varphi, \psi)$$ for all $\varphi, \psi \in \Ham(\S^2, \omega)$.  The Lipschitz constant  is $C_d = d$.  In particular, these invariants can be used to bound the Hofer distance from below.  

In view of the Hofer Lipschitz property, to prove our results, we will have to produce examples of Hamiltonian diffeomorphisms 
 whose invariants we can compute.  This will be done by studying  {\bf monotone twist Hamiltonians}, that is autonomous Hamiltonians $H : \S^2 \rightarrow \R$ of the form $$H(z, \theta) = \frac 12 h(z),$$  
where $h'\ge 0, h'' \ge 0, h(-1) = h'(-1) = 0$;  we developed a combinatorial model in our previous work \cite{CGHS} which can be used to compute the $c_d$ for Hamiltonians like this under the additional technical assumption that $h'(1) \in \N.$
For monotone twist Hamiltonians, the invariant $\zeta_d$ has a beautiful expression.

\begin{prop}\label{prop:zeta-d} For any  Hamiltonian  $H$ as above we have
\[\zeta_d(H)=\tfrac 12\sum_{i=1}^dh\left(-1+\frac{2i}{d+1}\right).\]
\end{prop}

In other words, $\zeta_d$ is the sum of the values of $H$ on $d$ equally distributed  horizontal circles.    
 We  learn from the above proposition that $\zeta_d(H)$ is at least as large as the value $H$ takes on each of the $d$ circles $\mathcal{C}_i = \{(z,\theta) : z = -1 + \frac{2i}{d+1}\}$, where $ i \in \{1, \ldots, d\}$.  This bears some resemblance to the notion of {\bf heaviness} of equators introduced in the works of Entov-Polterovich \cite{EP09}.  What is surprising is that the circles $\mathcal{C}_i$ are all displaceable for $d \ge 2$, while heaviness of a set, as defined in \cite{EP09}, implies that the set is not diplaceable by Hamiltonian diffeomorphisms.  Sensitivity to the  displaceable circles $\mathcal{C}_i$ is the 
distinguishing feature of our invariant $\mu_d, \zeta_d$ which powers our applications to the Hofer geometry of the kernel of Calabi.

 \subsection*{$C^0$ continuity and non-simplicity of $\Homeo_0(\S^2, \omega)$} 

 To prove Theorem \ref{theo:non-simplicity}, we need invariants which are continuous with respect to the $C^0$ topology.  The invariants $c_d$ and  $\mu_d$, while useful, are not in general $C^0$ continuous.  We remedy this by taking certain linear combinations of the $c_d$ to define $C^0$ continuous invariants
$$\eta_d :  \Ham(\S^2, \omega) \rightarrow \R.$$

Not only are these invariants $C^0$ continuous, but also they extend continuously to $\Homeo_0(\S^2, \omega)$.  Moreover, they are also Hofer Lipschitz.  We summarize the properties of the $\eta_d$ in Proposition \ref{prop:eta_d_properties}.   

\subsection{Relationship with previous work}
As mentioned above, in our previous work we used PFH to define spectral invariants for compactly supported area-preserving diffeomorphisms and homeomorphisms of the two-disc.  For all of the applications discussed here, we need to rework this theory over the two-sphere.  In the disc case, we could assume that the maps were generated by a Hamiltonian that vanishes near the boundary of the disc.  This is no longer possible, so new ideas are needed.

One idea here, familiar to specialists, see for example  \cite{EP03, Oh05, Schwarz}, is to attempt to work with mean-normalized Hamiltonians.  A careful analysis shows that this gives invariants which are well-defined on $\widetilde{\Ham}$; then, after homogenization as in the previous section, we can obtain invariants of $\Ham$.  These invariants would be enough to prove the theorems in \ref{sec:large}.  However, as stated above, they are not $C^0$ continuous, and so can not be used to study the algebraic structure of the homeomorphism group.  This is where the $\eta_d$, defined by taking a difference of spectral invariants, come in.  The crucial insight for this, which was initially surprising to us, is that the $c_d$ for even $d$ descend from $\widetilde{\Ham}$ to $\Ham$. 

\subsection*{Acknowledgments}
We thank Yves de Cornulier, Bertrand R\'emy, and Rich Schwartz for patiently responding to our questions concerning the beautiful subject of large-scale geometry.  We also thank Leonid Polterovich and Egor Shelukhin for very helpful correspondence concerning their work \cite{Pol-Shel2}, see Remark~\ref{rmk:ps}.  We thank Mohammed Abouzaid and Fr\'ed\'eric Le Roux for their comments on an earlier version of the article.  We also thank Dusa McDuff for very helpful communications about the perfectness question and Corollary \ref{non-perfectness-R2}. Lastly, we thank the anonymous referee for 
a careful reading and valuable remarks.

This article was written while DCG was at the Institute for Advanced Study, supported in part by the Minerva Research Foundation and the National Science Foundation.  DCG is extremely grateful to the institute for providing such a fantastic environment for conducting this research.  DCG also thanks the NSF for their support under agreement DMS $1711976$.  This project is an outgrowth of research that started in the summer of $2018$ when DCG was an ``FSMP Distinguished Professor" at the Institut Math\'{e}matiques de Jussieu-Paris Rive Gauche (IMJ-PRG). DCG is grateful to the Fondation Sciences Math\'{e}matiques de Paris (FSMP) and IMJ-PRG for their support.
 
This project has received funding from the European Research Council (ERC) under the European Union’s Horizon 2020 research and innovation program (grant agreement No. 851701) and from Agence Nationale de la Recherche (ANR project ``Microlocal'' ANR-15-CE40-0007).

 \section{Preliminaries}\label{sec:prelim}
 In this section we fix our notation and  introduce the necessary background on symplectic geometry and periodic Floer homology.
 
 \subsection{Recollections}\label{sec:prelim_symp}

Here we recall some basic facts about symplectic geometry and the Hofer distance.

Let $(M,\omega)$ be a symplectic manifold.   Let $H \in C^{\infty}(\S^1 \times M)$ be a Hamiltonian; if $M$ happens to be non-compact, then we consider only compactly supported Hamiltonians.  We can think of such  $H$ as a family of functions $H_t$ on $M$, depending on time; we think of $\S^1$ as parametrized by $0 \le t \le 1$.  Such $H$ gives rise to a possibly time-varying vector field $X_{H_t}$ on $M$, called the {\bf Hamiltonian vector field}, defined by 
\[ \omega(X_{H_t},\cdot) = d H_t.\]
The flow of $X_{H_t}$ is called the {\bf Hamiltonian flow} and is denoted $\varphi_H^t$.  The set of time-$1$ maps of Hamiltonian flows is called the set of {\bf Hamiltonian diffeomorphisms} of $M$ and denoted $\Ham(M,\omega);$ it forms a subgroup of the symplectomorphisms of $(M, \omega)$.   We can define the {\bf Hofer norm} $||\varphi||$ of any $\varphi \in \Ham(M,\omega)$ as follows.  First, to a Hamiltonian $H \in C^{\infty}(\S^1 \times M)$, we associate the norm
\[ \| H \|_{1,\infty} := \int_0^1 \left( \max_{M}(H_t) - \min_{M}(H_t) \right) dt.\]
We then define
\[ \| \varphi \| := \inf \lbrace \| H \|_{1,\infty} : \varphi = \varphi^1_H \rbrace.\]
The above quantity is invariant under conjugation, i.e. $\| \psi^{-1} \varphi \psi  \| = \| \varphi \|$.  This follows from the fact that $\varphi^t_{H\circ\psi} = \psi^{-1} \varphi^t_H \psi$; see \cite[Sec. 5.1, Prop. 1]{hofer-zehnder}, for example. 

Finally, we can define a metric on $\Ham(M,\omega)$, the {\bf Hofer metric}, by
\[ d_H(\varphi,\psi) = \| \varphi^{-1} \circ \psi \| .\]
As mentioned above, this yields a nondegenerate, bi-invariant metric, which is quite remarkable given the noncompactness of $\Ham$.  Non-degeneracy is what is difficult to prove and it was established by  Hofer for $\R^{2n}$ \cite{Hofer-metric}, by Polterovich for rational symplectic manifolds \cite{Polterovich93}, and by Lalonde-McDuff in full generality \cite{Lalonde-McDuff}.

The bi-invariance of Hofer's distance also implies the following identities 
\begin{align}
  d_H(\varphi_1 \varphi_2, \psi_1 \psi_2) &\leq d_H(\varphi_1, \psi_1) + d_H(\varphi_2, \psi_2),\label{eq:Hofer-identity1}\\
  d_H(\varphi, \psi^{-1} \varphi \psi) &\leq 2 d_H(\psi, \id).\label{eq:Hofer-identity2}
\end{align}
Indeed (\ref{eq:Hofer-identity1}) follows from (\ref{eq:norm-and-join}) below and (\ref{eq:Hofer-identity2}) is proved as follows:
  \[ d_H(\varphi, \psi^{-1} \varphi \psi) = \|\varphi^{-1}\psi^{-1} \varphi \psi\|\leq \|\varphi^{-1}\psi^{-1} \varphi\|+\| \psi\|= 2 d_H(\psi, \id).\]
Now let $M = \S^2 = \lbrace (x,y,z) \in \mathbb{R}^3 : x^2 + y^2 + z^2 = 1 \rbrace.$  This has a symplectic form $\omega := \frac{1}{4 \pi} d \theta \wedge dz,$ where $(\theta,z)$ are cylindrical coordinates.  We let $\Diff(S^2,\omega)$ denote the set of smooth diffeomorphisms $\varphi$, such that $\varphi^* \omega = \omega$.  In fact, $\Diff(\S^2,\omega) = \Ham(\S^2,\omega).$  The Hofer geometry of $\Diff(S^2,\omega)$, with this identification implied, will be the topic of study in the present work.  
We recall, for later use,  that the fundamental group of $\Ham(\S^2, \omega)$ is $\Z/2\Z$
and is generated by $\mathrm{Rot}$, the full rotation around the North-South axis of the sphere; for a proof of this see, for example, \cite{Polterovich2001}; the Hamiltonian $H(\theta, z) = \frac{1}{2} z$ generates this full rotation.

We will denote the universal cover of $\Ham(\S^2, \omega)$ by $\widetilde{\Ham}(\S^2, \omega)$.  This can be described as the set of 
Hamiltonian paths, considered up to homotopy relative to endpoints; here, by a {\bf Hamiltonian path}, we mean a path of Hamiltonian diffeomorphisms 
$\{\varphi^t, 0\le t \le 1\}$. This is a two-fold covering, by  the discussion in the previous paragraph.

We next recall the {\bf displacement energy} of a subset $A \subset \S^2$.  This is by definition the quantity
\[e(A) :=\inf\{\|\phi\|: \phi(A)\cap \overline{A}=\emptyset\}.\]
It is known that for a disjoint union of closed discs, each with area $a$ and 
whose union covers less than half the area of the sphere, the displacement energy is $a$. We will need the following lemma in Section \ref{sec:proof_embedding}.

\begin{lemma}\label{lem:transport_energy}
Let $D, D' \subset \S^2$ be two disjoint closed discs of equal area.  Then,
$$\inf\{\Vert \phi\Vert : \phi(D) = D'\} = \area(D).$$
\end{lemma}
\begin{proof} Let us denote $a:=\area(D)$ and $E:=\inf\{\Vert \phi\Vert : \phi(D) = D'\}$. It follows from the above discussion on displacement energy that $E\geq a$.

  For the reverse inequality, note that the same discussion also implies that for any $\eps>0$, there exists $\psi\in\Ham(\S^2,\omega)$, with $\|\psi\|<a+\eps$ and $\psi(D)\cap D=\emptyset$.  Since $\psi(D)$ and $D'$ have the same area and are both contained in $\S^2\setminus D$, there exists a Hamiltonian diffeomorphism $\chi$, supported in $\S^2\setminus D$ which maps $\psi(D)$ onto $D'$. The assumption on the support implies that $\chi\inv(D)=D$. We now pick $\phi=\chi\psi\chi\inv$. We see that $\phi(D)= D'$ and by conjugation invariance 
of the Hofer norm we have $\|\phi\|=\|\psi\|<a+\eps$.  Since such a diffeomorphism $\phi$ may be found for any $\eps>0$, this shows the reverse inequality $E\leq a$.  
\end{proof}

Next, we review the definition of the {\bf Calabi homomorphism}  \[\Cal:\Ham_U(\S^2,\omega)\to \R,\]
alluded to in the introduction.  Recall that, for proper open $U \subsetneq \S^2$, we denote by $\Ham_U(\S^2,\omega)$ the subgroup of $\Ham(\S^2, \omega)$ consisting of Hamiltonian diffeomorphisms which are supported in $U$.  Given $\varphi \in \Ham_U(\S^2,\omega)$, define
\begin{equation}\label{eqn:def_calabi}
\Cal(\varphi) = \int_{\S^1} \, \int_{\S^2} H(t, \cdot) \, \omega \, dt,
\end{equation}
where $H\in C^\infty(\S^1 \times \S^2)$ is any Hamiltonian supported in $U$ whose time--1 flow is $\varphi$.  It is well-known that $\Cal(\varphi)$ does not depend on the choice of $H$ and, moreover, $\Cal:\Ham_U(\S^2,\omega)\to \R$ is a group homomorphism; see \cite{calabi, McDuff-Salamon} for further details.

In Section \ref{sec:nonsimp}, we will also want to consider the group $\Homeo_0(\S^2,\omega)$ of {\bf area and orientation preserving homeomorphisms} of $\S^2$.  This is defined to be the group of homeomorphisms of $\S^2$, preserving the measure induced by $\omega$, in the component of the identity.  It has a distance $d_{C^0}$, called the $C^0$ distance, defined by picking a Riemannian metric $d$ on $\S^2$, and defining $$\displaystyle d_{C^0}(\varphi,\psi) = \sup_{x\in M} d(\varphi(x), \psi(x) ).$$ We remark for later use that $\Diff(\S^2, \omega)$ sits densely in the $C^0$ distance in $\Homeo_0(\S^2, \omega)$.

\subsection{The spectrum}
\label{sec:spec}
We now recall the {\bf action spectrum}, defined in \cite[Section 2.5]{CGHS}.  Let $H \in C^{\infty}(\S^1 \times \S^2).$ Recall the {\bf action functional} associated to $H$
\begin{equation}
\label{eqn:afun} 
\mathcal{A}_H(z,u) = \int^1_0 H(t,z(t)) dt + \int_{D^2} u^* \omega,
\end{equation}   
defined for capped loops $(z,u)$.  The critical points of $\mathcal{A}_H$ are pairs $(z,u)$, where $z$ is a $1$-periodic orbit of $\varphi_H^t,$ and the  
set of 
associated critical values 
is called the {\bf action spectrum} $\text{Spec}(H)$ of $H$.  The forthcoming PFH spectral invariants will take values in the {\bf order $d$} action spectrum of $H$, defined by
\[ \text{Spec}_d(H) := \cup_{k_1 + \ldots + k_j = d} \, \Spec(H^{k_1}) + \ldots + \Spec(H^{k_j}),\]
where $H^k$ denotes the $k$-fold composition of $H$ with itself.  Here, the {\bf composition} is defined by
  \[   (G \# H)(t,x)=
     \begin{cases}
       2\rho'(2t)H_{\rho(2t)}(x),&\quadif t\in[0,\tfrac12],\\
       2\rho'(2t-1)G_{\rho(2t-1)}(x),&\quadif t\in[\tfrac12, 1],
     \end{cases}
 \]  
 where $\rho: [0,1] \to [0,1]$ is a fixed non-decreasing smooth function which is equal to $0$ near $0$ and equal to $1$ near $1$. Note that we do not need $H$ and $G$ to be one-periodic to define the composition, 
 and even if they are not one-periodic, $G \# H$ will still be, since it is zero for $t$ close to $0$ and $1$.
The time $1$-map of $G \# H$ is $\varphi_G^1 \circ \varphi^1_H.$  %
 Note that for any Hamiltonians $G_1, G_2, H_1, H_2$, we have
   \begin{equation}
     \label{eq:norm-and-join}
    \|G_1\# H_1-G_2\# H_2\|_{1,\infty}= \|G_1-G_2\|_{1,\infty}+\| H_1- H_2\|_{1,\infty}.
   \end{equation}

We state here some of the properties of the order $d$ action spectrum which will be used in the following sections. Recall that $H \in C^\infty(\S^1 \times \S^2)$ is said to be {\bf mean-normalized} if $\int_{\S^2} H(t, \cdot) \omega = 0$ for all $t \in \S^2$.  Two Hamiltonians $H_0, H_1$ are said to be {\bf homotopic} if there exists a smooth path of Hamiltonians connecting $H_0$ to $H_1$ such that $\varphi^1_{H_0}= \varphi^1_{H_s} = \varphi^1_{H_1}$ for all $s\in[0,1]$.  In other words, the Hamiltonian paths 
$\{\varphi^t_{H_0}\}$ and $\{\varphi^t_{H_1}\}$, for $0 \leq t \leq 1$, coincide as elements of the universal cover $\Hamtilde(\S^2, \omega)$. Here is a list of properties of $\spec_d$ which will be needed.
\begin{enumerate}[(i)]
\item {\bf Symplectic invariance:} $\spec_d(H \circ \psi) = \spec_d(H)$, for all $H \in C^\infty(\S^1 \times \S^2)$ and $\psi \in \Ham(\S^2, \omega)$.
\item {\bf Homotopy invariance:} If $H_0, H_1$ are mean-normalized and homotopic, then $\spec_d(H_0) = \spec_d(H_1)$.%
\item {\bf Measure zero:} $\spec_d(H)$ is of measure zero.
\end{enumerate}

The above properties are well-known in the case of $\spec(H)$, that is when $d=1$; see for example \cite{Oh05}.  It is not difficult to see 
that the two initial properties follow from the case $d=1$: Symplectic invariance follows from the identity $(H\circ\psi)^k = H^k \circ \psi$, for any $k\in \N$, and Homotopy invariance is a consequence of the fact that $H_0^k, H_1^k$ are mean-normalized and homotopic, for any $k\in \N$, if $H_0$ and $H_1$ are.
 As we will now explain, the third property also follows from the $d=1$ case.  As a consequence of the definition of $\spec_d(H)$, it is sufficient to prove that the set $\Spec(H^{k_1}) + \ldots + \Spec(H^{k_j})$ is of measure zero, for any choice of $k_1, \ldots, k_j$ with the property that $k_1 + \ldots +k_j =d$.  To that end, let $(M, \omega \oplus \ldots \oplus \omega)$ be the symplectic manifold obtained by taking the $j-$fold product of $(\S^2, \omega)$ and consider the Hamiltonian  $F:\S^1 \times M \rightarrow\R$ defined by
$$F(t,x_1, \ldots, x_j) = H^{k_1}(t,x_1) +\ldots + H^{k_j}(t,x_j).$$  We conclude that   $\Spec(H^{k_1}) + \ldots + \Spec(H^{k_j})$ has measure zero by observing that it coincides with the set $\spec(F)$ which we know has measure zero.  %

\subsection{Definition of PFH}\label{sec:prelim_PFH}

We now recall the definition of periodic Floer homology (PFH), from for example \cite{Hutchings-Sullivan-Dehntwist}, which is a tool that will be central in our work.  While PFH can be defined over any surface, for simplicity we consider the case where our surface is $\S^2$,
which is the only case that is relevant for the present work.

We start with some preliminaries.  Let $\varphi \in \Diff(\mathbb{S}^2,\omega)$. Given $\varphi$, we can define the {\bf mapping torus}
\[ Y_{\varphi} := \mathbb{S}^2 \times [0,1]_t / \sim, \quad (x,1) \sim (\varphi(x),0).\]
This has a natural vector field $R := \partial_t$, which we call the {\bf Reeb} vector field, a natural one form $dt$, and a natural two-form $\omega_{\varphi}$ induced from the area form $\omega$.  The pair $(dt,\omega_{\varphi})$ is a {\bf stable Hamiltonian structure} in the sense of for example \cite{BEHWZ, CiMo, HT-stable, Wendl-Notes}.  The manifold $Y_{\varphi}$ has a plane field $\xi$ defined to be the vertical tangent bundle for the fibration $\pi: Y_{\varphi} \to \S^1$. 

We will be interested in closed integral curves
\[ \alpha: \mathbb{R}/T\mathbb{Z} \to Y_{\varphi},\]
of $R$, modulo reparametrization of the domain, which we call {\bf closed orbits}; we can identify an embedded closed orbit with its image.    A closed orbit $\alpha$ has an integral {\bf degree} $d(\alpha) := \pi_*[\alpha] \in H_1(\S^1) = \mathbb{Z}.$   The {\bf linearized return map} $P_{\alpha}$ for a closed orbit $\alpha$ is defined for any $p \in \alpha$ as the linearization of the time $T$ flow of $R$ on $\xi|_p$.  A closed orbit is called {\bf nondegenerate} if $1$ is not an eigenvalue of the linearized return map; a nondegenerate closed orbit is called {\bf hyperbolic} if the eigenvalues of $P_{\alpha}$ are real and {\bf elliptic} if the eigenvalues lie on the unit circle; these definitions do not depend on the choice of $p$.  
  
Define an {\bf orbit set} $\alpha := \lbrace (\alpha_i,m_i) \rbrace$ to be a finite set, where the $\alpha_i$ are distinct embedded closed orbits of $R$, and the $m_i$ are positive integers.   The {\bf degree} of the orbit set $\alpha$ is the sum of the degrees of the $\alpha_i$.    
The map $\varphi$ is {\bf $d$-nondegenerate} if every closed orbit %
with degree at most $d$ is nondegenerate; this is a generic condition.  %
A degree $d$ orbit set for a $d$-nondegenerate $\varphi$ is called {\bf admissible} if $m_i = 1$ whenever $\alpha_i$ is hyperbolic.  

Let $X = \mathbb{R}_s \times Y_{\varphi}$.  This has a natural symplectic form $$\omega = ds \wedge dt + \omega_{\varphi}.$$  
The pair $(X, \omega)$ is called the {\bf symplectization} of $Y_{\varphi}.$  Recall that an {\bf almost complex structure} on $X$ is a smooth bundle map $J: TX \to TX$ such that $J^2 = -1$.  A {\bf $J$-holomorphic curve} in $X$ is a map $u: (\Sigma,j) \to (X,J)$, satisfying the equation
\[ du \circ j = J \circ du.\]
Here, $\Sigma$ is a closed (possibly disconnected) Riemann surface, minus a finite number of punctures, and the map $u$ is assumed asymptotic to Reeb orbits near the punctures, see for example \cite{Hutchings-Notes} for the precise definition. 

The {\bf periodic Floer homology} $PFH(\S^2,\varphi,d)$ is the homology of a chain complex $PFC(\S^2,\varphi,d)$.  The chain complex $PFC(\S^2,\varphi,d)$ is freely generated over $\mathbb{Z}_2$ by admissible orbit sets $\alpha$ of degree $d > 0$.  The chain complex differential $\partial$ counts $J$-holomorphic curves in $X$, for generic admissible $J$; here, an almost complex structure is called {\bf admissible} if it preserves $\xi$, is $\mathbb{R}$-invariant, sends $\partial_s$ to $R$, and its restriction to $\xi$ is tamed by $\omega_{\varphi}.$  More precisely, 
\[ \langle \partial \alpha, \beta \rangle = \# \mathcal{M}_J^{I = 1}(\alpha,\beta),\]
where $I$ denotes the ECH index, defined below, we are considering curves in $X$ up to equivalence of currents and modulo translation in the $\mathbb{R}$ direction, and $\#$ denotes the mod $2$ count.

It is shown in \cite{Hutchings-TaubesI,Hutchings-TaubesII} that\footnote{More precisely, \cite{Hutchings-TaubesII} 
proves that the differential in embedded contact homology squares to zero.  As pointed out in \cite{Hutchings-TaubesI} and \cite{Lee-Taubes} this proof carries over, nearly verbatim, to our setting.}  that $\partial^2 = 0$, so the homology is well-defined; it is shown in \cite{Lee-Taubes} that it agrees with a version of Seiberg-Witten Floer cohomology and in particular is independent of $\varphi$.%

To define spectral invariants, we will want to use a {\bf twisted} version of PFH, denoted $\widetilde{PFH}(\mathbb{S}^2,\varphi,d)$; 
as we will see in \ref{sec:action}, the twisted PFH carries a natural action filtration which we will use to define the spectral invariants.  To define twisted PFH, let $\gamma$ be any degree $1$ cycle in $Y_{\varphi}$, transverse to $\xi$;
choose a homotopy class of trivializations $\tau_0$ on $\xi|_{\gamma}$.  The twisted PFH chain complex $\widetilde{PFC}$ is generated by pairs $(\alpha,Z)$, called {\bf twisted PFH generators}, where $\alpha$ is a degree $d$ admissible orbit set, and $Z \in H_2(Y_{\varphi},\alpha,d \gamma)$.
The differential counts $I = 1$ curves $C$ from $(\alpha,Z)$ to $(\beta,Z')$, namely curves $C \in \mathcal{M}_J^{I = 1}(\alpha,\beta)$, such that 
\[ [C] + Z' = Z.\]
For each $d$, there is a grading, defined below, which we call the {\bf $k$-grading}.  The homology is an invariant, and so can be computed, with the result that for $d \ge 0$ we have  
\begin{equation}\label{eq:PFH_sphere}
\widetilde{PFH}_*(\S^2,\varphi,d) = 
\begin{cases}
  \mathbb{Z}_2, & \text{if } *=d \text{ mod } 2, \\
  0 & \text{otherwise.}
\end{cases}
\end{equation}
The above identity can be proven via a direct  computation when $\varphi$ is taken to be an irrational rotation of the sphere; for more details see, for example,  \cite[Sec.\ 3.3]{CGHS}.  
We now define the ECH index $I$, and the grading $k$.   

The ECH index $I$ depends only on the relative homology class $A \in H_2(Y_{\varphi},\alpha,\beta)$ between two orbit sets.  We have
\begin{equation}
\label{eqn:echindex}
I(A) = c_{\tau}(A) + Q_{\tau}(A) + CZ_{\tau}^I(A),
\end{equation}
where $\tau$ denotes a homotopy class of trivializations of $\xi$ over all Reeb orbits, $c_{\tau}(A)$ denotes the {\bf relative Chern class} of $\xi$ restricted to $A$, $Q_{\tau}(A)$ denotes the {\bf relative self-intersection}, and $CZ^I_{\tau}$ denotes the total Conley-Zehnder index.  We will not need the precise definitions of these terms in the present work, so we omit them for brevity, referring the reader to \cite{Hutchings-index} for the details.  

We can define the promised $k$ grading.  The definitions of the relative Chern class and relative self-intersection extend verbatim to relative homology classes $A \in H_2(Y_{\varphi},\alpha,d \gamma)$, once a trivialization $\tau$ over the simple orbits in $\alpha$ and a trivialization $\tau_0$ over $\gamma$ has been chosen.  With the preceding understood, 
we now define   
\[ k(\alpha,Z) := c_{\tau,\tau_0}(Z), + Q_{\tau,\tau_0}(Z) + CZ^I_{\tau}(\alpha).\]
To simplify the notation, we will denote $k(\alpha,Z)$ by $I(Z)$ below.

\section{The spectral invariants}
\label{sec:newspec}

We now use the twisted PFH to define various invariants.  We begin by summarizing for the reader what will be done in this section.  

To set the stage for what is coming, it is helpful to recall what was done in \cite[Sec.\ 3.4]{CGHS}.  There, we defined spectral invariants $c_{d,k}(H)$ for $H\in \mathcal{H}$ where
\begin{align*}
\mathcal{H} := \{H \in C^{\infty}(\S^1 \times \S^2) : &\  \varphi^t_H(p_-) = p_- , H(t, p_-) = 0,  \forall t \in [0,1],
\\  & -\tfrac14 < \rot( \{\varphi^t_H\}, p_- ) < \tfrac14  \},
\end{align*}
where $\rot( \{\varphi^t_H\} , p_-)$ is the rotation number of the isotopy $\{\varphi^t_H\}_{t \in [0,1] }$ at $p_-$.  It was shown in addition that these invariants depend only on the time $1$-map.  Spectral invariants for compactly supported disc maps were then defined by identifying the disc with the northern hemisphere.  

Our goal now is to define spectral invariants for all $H \in C^{\infty}(\mathbb{S}^1 \times \mathbb{S}^2)$  and to find invariants that depend only on $\varphi \in \Ham(\S^2,\omega)$, rather than on a choice of generating Hamiltonian.  Here is how we do this.  First we extend the procedure in \cite{CGHS} from $H \in \mathcal{H}$ to arbitrary $H$ to get invariants $c_{d,k}$, defined when $k$ and $d$ have the same parity.  These $c_{d,k}$ extend the $c_{d,k}$ from our previous work: that is, if $H \in \mathcal{H} \subset C^{\infty}(\S^1 \times \S^2)$, then the definition of $c_{d,k}(H)$ here agrees with that in \cite{CGHS}.   Similarly to our previous work, we can then define $c_d := c_{d,-d}$.  This choice of $k = -d$ is not quite canonical, see Remark~\ref{rmk:noncanonical}, but is convenient and suffices for our purposes: what is crucial is that $c_d(0) = 0$.

As alluded to in the introduction, these $c_d$ are in general {\em not} invariants of the time $1$-map, and so are not well-suited on their own for proving our main theorems.  
However, we can use the $c_{d}$ to form new invariants.  First, we show that the $c_d$ for even $d$ are invariants when we restrict to mean-normalized Hamiltonians; similarly, the homogenizations %
$\mu_d, \zeta_d$  are also invariants restricted to mean-normalized Hamiltonians.  None of these invariants are $C^0$ continuous, so we use a linear combination of the $c_d$ for $d$ even to define another sequence $\eta_d$.

Thus, to summarize for the ease of the reader, the main product of this section are invariants $c_d$ and $\eta_d$ defined for $d$ even, and $\mu_d, \zeta_d$ defined for all $d$, together with proofs of their properties that we will need.  The $\mu_d$ and $\zeta_d$ are related by the formula \eqref{eqn:mudzeta}.  The $\mu_d$ are used to prove Theorem~\ref{theo:QI-ker-Cal}, while the $\eta_d$ are used to prove Theorem~\ref{theo:non-simplicity}; the $c_d$ are used to construct the $\mu_d$ and the $\eta_d$.

\subsection{Invariants for Hamiltonians}
We begin by introducing PFH spectral invariants $c_{d,k}(H)$ for Hamiltonians $H \in C^\infty(\S^1 \times \S^2)$.   This requires first recalling a construction of Hutchings for assigning a spectral invariant to every nonzero twisted PFH class. 

\subsubsection{The nondegenerate case}
\label{sec:action}

A Hamiltonian $H \in C^{\infty}(\mathbb{S}^1 \times \mathbb{S}^2)$ is called {\bf $d$-nondegenerate} if its time-$1$ flow $\varphi = \varphi^1_H$ is $d$-nondegenerate.
We now explain how to define PFH spectral invariants for $d$-nondegenerate Hamiltonians by extending the definition in \cite{CGHS} in a natural way.

 We begin by explaining the aforementioned construction of Hutchings for assigning a spectral invariant to a nonzero twisted PFH class.  A twisted PFH generator has an {\bf action} defined by
\[ \mathcal{A}(\alpha,Z) = \int_Z \omega_{\varphi}.\]
The differential decreases the action, see for example \cite[Sec 3.3]{CGHS}, so the action induces a filtration on the twisted PFH chain complex: we can define $\widetilde{PFC}^L$ to be the subcomplex generated by twisted PFH generators with action no more than $L$.  %
Denote the homology of this complex by $\widetilde{PFH}^L$.  For  any nonzero class  $\sigma \in \widetilde{PFH}(\S^2,\varphi, d)$, we can now define $c_{\sigma}(\varphi,\gamma,\tau_0)$ to be the smallest $L$ such that $\sigma$ is in the image of the inclusion induced map
\[ \widetilde{PFH}^L \to \widetilde{PFH}.\]
We can think of this as the minimum action required to represent $\sigma$.

The number $c_{\sigma}(\varphi,\gamma,\tau_0)$ depends on the choice of reference cycle $\gamma$ and trivialization $\tau_0$ over $\gamma$; we will now define the {\bf PFH spectral invariants} associated to a $d$-nondegenerate Hamiltonian $H$ by using the Hamiltonian flow to fix a natural reference cycle.

To make this precise, let $H$ be a $d$-nondegenerate Hamiltonian and write $\varphi = \varphi_H^1.$
Consider the trivialization
\begin{equation} \label{eq:trivialization}
\begin{split}
  \Psi_H :\;  &  \S^1 \times \S^2 \rightarrow Y_{\varphi} \\
 (t,x) \mapsto &  \left( (\varphi^{t}_H)^{-1} (x), t \right).
\end{split}
\end{equation}
Define $\gamma_H = \Psi_H ( \S^1 \times \{p_-\} )$.  This is trivialized by the pushforward $\tau_H$ of an $\mathbb{S}^1$-invariant trivialization over $p_-$.  We will now use the twisted PFH chain complex for $Y_{\varphi}$, with respect to the reference cycle $\gamma_H$, 
to define the spectral invariants.  

Assume first that $H$ vanishes at $p_-$ for all time.  For each $d\in \N$, we define
\[ c_{d,k}(H)   := c_{\sigma}(\varphi^1_H,\gamma_H,\tau_H), \quad d \equiv k \quad \text{mod 2},\]
where $\sigma$ is the unique nonzero class in $\widetilde{PFH}_k(\S^2,\varphi,d)$.  We emphasize that, even fixing the Hamiltonian diffeomorphism, this  can and will depend on $H$, since the trivialized reference cycle $\gamma_H$ does.
We note that for such an $H$, 
\begin{equation}
\label{eqn:spec1} 
c_{d,k}(H) = \mathcal{A}(\alpha,Z),
\end{equation}
for some twisted PFH generator $(\alpha,Z).$  Indeed, as explained in \cite[Sec.\ 3.3]{CGHS} this follows from the fact that the subset $\lbrace \mathcal{A}(\alpha,Z)  : (\alpha,Z) \in \widetilde{PFC}(\varphi,d) \rbrace$ $\subset \mathbb{R}$ is discrete, as under our nondegeneracy assumption there are only finitely many orbit sets of degree $d$. 

Finally, for arbitrary $H$ we reduce to the case of $H$ vanishing at $p$ by demanding that the {\bf Shift property}, stated in Proposition~\ref{prop:more} below, hold.  This says that
\begin{equation}
\label{eqn:shift1} 
c_{d,k}(H+h) = c_{d,k}(H) + d \int^1_0 h(t) dt,
\end{equation}
when $h: \mathbb{S}^1 \to \mathbb{R}$ is any function.  %

In principle, $c_{d,k}(H)$  could depend on the choice of admissible $J$, but we will see by the Monotonicity property below that it does not.

\subsubsection{Key properties }

We now prove that the PFH spectral invariants have the following key properties and extend to all, possibly degenerate, Hamiltonians.

\begin{theo}
\label{thm:PFHspec_initial_properties}
The PFH spectral invariant $c_{d,k}(H)$ admits a unique extension 
to all $H \in C^{\infty}(\S^1 \times \S^2)$ such that the extended spectral invariant $$c_{d,k} : C^{\infty}(\S^1 \times \S^2)  \rightarrow \R$$ satisfies the following properties.
\begin{enumerate}
\item Continuity:  For any $H, G \in C^{\infty}(\S^1 \times \S^2)$, we have 
$$ d \int_{\S^1} \min(H_t -G_t) \, dt \le c_{d,k} (H) - c_{d, k} (G)  \leq  d \; \int_{\S^1} \max(H_t -G_t) \, dt.$$ 
\item Spectrality: $c_{d, k}(H) \in \Spec_d(H)$.
\end{enumerate}
\end{theo}

Before giving the proof, we note that the second item of the theorem implies that if $H, G$ vanish at $p_-$,  then
\begin{equation}
\label{eqn:hofcont}
| c_{d,k}(H) - c_{d,k}(G) | \le d \|H - G\|_{1,\infty},
\end{equation} 
which is an alternative variant of the Hofer continuity property.  

\begin{proof}

The proof proceeds along similar lines as [CG-H-S, Thm. 3.6].  

{\em Step 1: Reducing to the $d$-nondegenerate case.} We now assume that the theorem has been proved for $d$-nondegenerate $H$, and explain how this implies the result for all $H$.  Given any $H$,
take any sequence of $d$-nondegenerate $H_i$ which $C^2$ converges to $H$, and define
\begin{equation}
\label{eqn:degdefn}
c_{d,k}(H) = \lim_{i \to \infty} c_{d,k}(H_i).
\end{equation}
This limit exists, and does not depend on the choice of approximating $H_i$, due to the Continuity property with $H = H_i$ and $G = H_j$.
   The same inequality implies that the extension from $d$-nondegenerate $H$ is unique as claimed; the Continuity and Shift properties for $d$-nondegenerate $H$ imply these properties for all $H$.   Spectrality for $d$-nondegenerate $H$ implies Spectrality for all $H$ by Arzela-Ascoli. 

\medskip

{\em Step 2: Reducing to Hamiltonians that vanish at $p_-$.}  %

It remains to prove Continuity
and Spectrality 
in the nondegenerate case.  

We now show that by using the Shift property \eqref{eqn:shift1}, it suffices to prove these properties for Hamiltonians vanishing at $p_-$. 
We begin with Continuity.  %
Consider arbitrary $H, G$.  Then, we can write
\begin{equation}
\label{eqn:decomp}
H = \tilde{H} + h, \quad G = \tilde{G} + g,
\end{equation}
where $h$ and $g$ are defined as the restriction of $H, G$ to $p_-$, and $\tilde{H}, \tilde{G}$ vanish on $p_-$.  Then, by the Shift property,
\[ c_{d,k}(H) - c_{d,k}(G) = c_{d,k}(\tilde{H}) - c_{d,k}(\tilde{G}) + d \int_{\S^1} (h(t) - g(t)) \, dt.\]
Thus, if Continuity holds for $\tilde{H}$ and $\tilde{G}$, then we have
\[ c_{d,k}(H) - c_{d,k}(G) \le d \int_{\S^1} \max( \tilde{H}_t - \tilde{G}_t )\,dt+ d \int_{\S^1} (h(t) - g(t)) \,dt.\]

 Now, since $h, g$ only depend on $t$, we have
 \[ \max( \tilde{H}_t - \tilde{G}_t ) = \max( H_t - G_t ) + g(t) - h(t).\]
 
Combining this equality with the previous inequality proves the rightmost inequality required for Continuity.    Similarly, if Continuity holds for $\tilde{H}$ and $\tilde{G}$, then we have
\[ c_{d,k}(H) - c_{d,k}(G) \ge d \int_{\S^1} \min( \tilde{H}_t - \tilde{G}_t )\,dt+ d \int_{\S^1} (h(t) - g(t))\,  dt,\]
and we know that
 \[ \min( \tilde{H}_t - \tilde{G}_t ) = \min( H_t - G_t ) + g(t) - h(t),\]
hence the leftmost inequality required for Continuity to hold.
 
 Similarly, if Spectrality holds for $\tilde{H}$ in \eqref{eqn:decomp}, then it holds for $H$ by the Shift property, because the addition of $h$ does not change the set of critical points of $\mathcal{A}_H$, hence by \eqref{eqn:afun}, $\text{Spec}_d(H) = \text{Spec}_d(\tilde{H}) + d \int_{\S^1} h(t) \, dt.$ %

Thus, we can assume $H$ and $G$ vanish at $p_-$.  

\medskip

{\em Step 3.  Continuity when $H$ and $G$ vanish at $p_-$}. 

Under \eqref{eq:trivialization}, the stable Hamiltonian structure $(dt,\omega_{\varphi})$ is of the form $(dt, \omega + dH\wedge dt),$ and $R = \partial_t + X_H$.  The natural symplectic form on the symplectization $X = \mathbb{R} \times Y_\varphi$ under \eqref{eq:trivialization} is
\[ \omega_H = ds \wedge dt + \omega + dH \wedge dt,\]
where $s$ is the coordinate on $\mathbb{R}$.  We henceforth identify $Y_{\varphi}$ with $\S^1 \times \S^2$ using \eqref{eq:trivialization}, we implicitly identify orbit sets on $Y_{\varphi}$ with the corresponding orbit sets on $\S^1 \times \S^2$, and we identify the trivialized reference cycle $(\gamma_H,\tau_H)$ with the $\S^1$-invariant trivialized cycle $\gamma$ over $p_-$.  %

Given $H$ and $G$, we pick a function $\beta$, which is $0$ for sufficiently small $s$, $1$ for $s$ sufficiently large, and satisfies $1 + \beta'(H-G) > 0$, we define $K = G + \beta(s) (H - G)$ and we consider the form
\[ \omega_X = ds \wedge dt + \omega + d(K dt),\]
which is symplectic and agrees with $\omega_H$ for sufficiently positive $s$ and $\omega_G$ for sufficiently negative $s$.

The general theory of (twisted) PFH cobordism maps, as developed by Chen \cite{Chen}, guarantees a chain map $\Psi_{H,G}$ between the twisted PFH chain complexes for $H$ and $G$, counting ECH index zero
$J_X$-holomorphic buildings from $(\alpha,Z)$ to $(\beta,Z')$, and inducing an isomorphism, where $J_X$ is a fibration compatible almost complex structure on $X$, in the sense that it preserves the vertical tangent bundle and its $\omega_X$-orthogonal complement.  

So, given 
$d\geq 1$ and $k\in \Z$ of the same parity, let $(\alpha_1,Z_1) + \ldots + (\alpha_m,Z_m)$ be a cycle in $\widetilde{PFC}(\varphi^1_H,d)$ representing $\sigma_{d,k}$ with
\[ c_{\sigma_{d,k}}(\varphi^1_H) = \mathcal{A}(\alpha_1,Z_1) \ge \ldots \ge \mathcal{A}(\alpha_m,Z_m)\]
and let $(\beta,Z')$ be a generator in $\widetilde{PFC}(\varphi^1_G,d)$ with maximal action among the support of $\Psi_{H,G}( (\alpha_1,Z_1) + \ldots + (\alpha_m,Z_m) ).$ 

Thus, we have a $J_X$-holomorphic building $C$ from some $(\alpha_i,Z_i)$, which we will denote by $(\alpha,Z)$, to $(\beta,Z')$.  Since, just as in \cite{CGHS}, our argument only involves action and index considerations,  we can assume that $C$ consists of a single level, and we know that 
\[ Z' + [C] = Z,\]
as elements of $H_2(\mathbb{S}^1 \times \mathbb{S}^2, \alpha, d \gamma)$.  Hence, as $I([C]) = 0$, we must have $I(Z) = I(Z') = k$, so that 
\begin{equation}
\label{eqn:diff3}
c_{d,k}(\varphi^1_H) - c_{d,k}(\varphi^1_G) \ge \mathcal{A}(\alpha,Z) - \mathcal{A}(\beta,Z').
\end{equation}

We now claim the identity
\begin{equation}
\label{eqn:diff}
\mathcal{A}(\alpha,Z) - \mathcal{A}(\beta,Z') = \int_C \omega + dK \wedge dt + K' ds \wedge dt,
\end{equation}
where $K'$ denotes the derivative with respect to $s$, and for the rest of this section $dK$ denotes the derivative in the $\S^2$ direction.

The proof of this is just as in \cite[Lem. 3.8]{CGHS}. 
Indeed, as in the proof of \cite[Lem. 3.8]{CGHS},  we have
\[ \mathcal{A}(\alpha,Z) = \int_Z \omega + d(Hdt), \quad \mathcal{A}(\beta,Z') = \int_{Z'} \omega + d(Gdt),\]
and
\[ \int_C \omega = \int_Z \omega - \int_{Z'} \omega.\]
Moreover, $\int_C d(Kdt) = \int_Z d(Hdt) - \int_{Z'} d(Gdt)$, since $H, G$ vanish on $\gamma$. So, putting this all together, we have
\[ \mathcal{A}(\alpha,Z) - \mathcal{A}(\beta,Z') = \int_C \omega + d(Kdt),\]
hence \eqref{eqn:diff}.  

Moreover, we have $\int_C \omega + dK \wedge dt \ge 0$, since as in the proof of \cite[Lem. 3.8]{CGHS}, the form $\omega + dK \wedge dt$ is pointwise nonnegative along $C$, and so in fact we obtain
\begin{equation}
\label{eqn:diff2}
\mathcal{A}(\alpha,Z) - \mathcal{A}(\beta,Z') \ge \int_C K' ds \wedge dt,
\end{equation}
The argument in \cite[Lem. 3.8]{CGHS} also shows that $ds \wedge dt$ is pointwise nonnegative on $C$. 

Now we have
\[ \int_C K' ds \wedge dt = \int_C \beta'(s) (H-G) ds \wedge dt \ge \int_C \beta'(s) \min(H_t - G_t) ds \wedge dt,\]
since $ds \wedge dt$ is pointwise nonnegative along $C$.  We can evaluate the rightmost integral in the above equation by projecting to the $(s,t)$ plane; this projection has degree $d$, and $\int \beta' = 1$, so the above inequality in combination with \eqref{eqn:diff2} and \eqref{eqn:diff3} give the leftmost inequality required for Continuity.

To prove the other inequality, we switch the role of $H$ and $G$ in the above argument, and again combine the corresponding versions of \eqref{eqn:diff3} and \eqref{eqn:diff2} to get
\[ c_{d,k}(\varphi^1_G) - c_{d,k}(\varphi^1_H) \ge \int_C \beta'(s) (G - H) ds \wedge dt,\]
hence
\[c_{d,k}(\varphi^1_H) - c_{d,k}(\varphi^1_G) \le \int_C \beta'(s) (H-G) ds \wedge dt \le  \int_C \beta'(s) \max (H_t-G_t) ds \wedge dt,\]
where in the rightmost inequality we have used the fact that $ds \wedge dt$ is pointwise nonnegative.  We then project to the $(s,t)$ plane as above to obtain the rightmost inequality required for Continuity.

\medskip
{\em Step 4.  Spectrality when $H$ vanishes at $p_-$.}

Since $H$ vanishes at $p_-$, we know by \eqref{eqn:spec1} that any $c_{d,k}(H) = \mathcal{A}(\alpha,Z)$ for some twisted PFH generator $(\alpha,Z)$.  

Recall that  $\mathcal{A}(\alpha,Z)$ is the action of some relative homology class. %
We first construct a particular homology class $Z_{\alpha}$ from a periodic orbit $\alpha$, and show that the action of this class lies in the action spectrum.  More precisely,      
let $x$ be a $q$ periodic point of $\varphi=\varphi_H^1$, and pick a capping disk $u$ for the orbit $\gamma(t)=(\varphi_H^t(x))_{t\in[0,q]}$, such that $u(0,0) = p_-$. 
Equip the disc with polar coordinates $(\theta,\rho)$ with $\theta \in \mathbb{R}/q\mathbb{Z}$, $\rho \le 1$, and then consider the homology class $Z_{\alpha}$ represented by
\[ \mathbb{R}/q\mathbb{Z} \times [0,1] \to \S^1 \times \S^2, \quad (\theta,\rho) \to (\theta \hspace{1 mm} \text{mod 1}, u(\theta,\rho)).\]
We now compute %
  \begin{align*}
    \cal A(\beta, Z_\alpha)&=\int_{Z_\alpha}(\omega+dH\wedge dt) = \int_{Z_\alpha}(\omega+d(H dt) )\\
                          &= \int u^*\omega +\int_{\partial Z_\alpha} Hdt\\
                          &= \int u^*\omega +\int_0^q H_t(\gamma(t))dt \\
                          &=\cal A_H(\gamma, u) \in \spec_q(H),
  \end{align*}
where, in the third equality above, we have used the fact that $H$ vanishes at $p_-$.

Now, given an arbitrary $(\alpha,Z)$, write $\alpha = \lbrace (\alpha_i,q_i) \rbrace$.  We can write $Z = \sum Z_{\alpha_i} + y [\mathbb{S}^2].$  Then
  \begin{align*}
    \cal A(\alpha, Z)=y+\sum \cal A(\alpha_i, Z_{\alpha_i}). 
   \end{align*}
The right hand side of the above formula is an element of $\spec_d(H)$, since we can for example absorb the $y$ into the capping of any particular orbit.
Hence, $c_{d,k}(H)\in \spec_d(H)$.
\end{proof}

 We now collect some additional useful properties of the $c_{d,k}$.
 
 \begin{prop}
 \label{prop:more}
 
The spectral invariant $$c_{d,k} : C^{\infty}(\S^1 \times \S^2)  \rightarrow \R$$ satisfies:

\begin{enumerate}
\item Normalization: $c_{d,k}(0) = 0$ for $-d \le k \le d$ 
\item Monotonicity: Suppose that $H \leq G$.
Then,
\[ c_{d, k}(H) \leq c_{d, k} (G).\]

\item Shift: Let $h: \S^1 \to \R$ be a function of time.  Then, $$c_{d,k}(H + h) = c_{d,k}(H) + d \int_{\S^1} h(t) dt.$$ 

\item Symplectic invariance: $c_{d,k}(H\circ \psi) = c_{d,k}(H)$ for any $\psi \in \Ham(\S^2, \omega)$.

\item Homotopy invariance:  If  $H_0, H_1$ are mean-normalized and homotopic, 
then $c_{d,k}(H_0) = c_{d,k}(H_1)$ 

\item Support-control: If the support of $H$ is contained in a topological disc $D$ with $\area(D) < \frac{1}{d+1} $, and $-d \le k \le d$, then $|c_{d,k}(H)|\leq 2d \; \area(D)$. 
\end{enumerate}
\end{prop}

\begin{proof} Normalization follows from our previous work \cite[Thm. 3.6]{CGHS}, since as mentioned previously the $c_{d,k}$ extend the spectral invariants we defined there.  The Shift property is immediate from the definition.   The Monotonicity property follows formally from Continuity: indeed, by Continuity we have
\[ c_{d,k}(H) - c_{d,k}(G) \le d \int_{\S^1} \max(H_t - G_t)\, dt,\]
and so if $H \le G$ then the integrand in the above inequality is nonpositive, so that we obtain Monotonicity.

To prove Symplectic invariance, let $\psi_t$ be a Hamiltonian isotopy such that $\psi_0 = \id, \psi_1= \psi$. It is sufficient to show that the function $t \mapsto c_{d,k}(H\circ \psi_t)$ is constant.   To see this, recall from Section \ref{sec:spec} that $\spec_d(H\circ \psi_t) = \spec_d(H)$ and so the function $t \mapsto c_{d,k}(H\circ \psi_t)$, which is continuous by the Continuity property of Theorem \ref{thm:PFHspec_initial_properties}, takes values in the measure-zero set $\spec_d(H)$ and so it must be constant. 

The proof of Homotopy invariance is analogous.  Let $H_s, 0 \le s \le 1$, be a smooth path of mean-normalized Hamiltonians connecting $H_0$ to $H_1$.  Note that, by the Homotopy invariance of the action spectrum from Section \ref{sec:spec}, we have  $\spec_d(H_s) = \spec_d(H_0)$ for all $d\in \N$ and $s \in [0,1]$.  Then, the continuous function $s \mapsto c_{d,k}(H_s)$ is constant because it takes values in the measure zero set $\spec_d(H_0)$.  We conclude that $c_{d,k}(H_0) = c_{d,k}(H_1)$.

It therefore remains to prove Support-control.  The proof will rely on the following lemma.  We will say that a set $U$ is {\bf $d$-displaced} by a map $\Psi$ if the sets $U, \Psi(U),\ldots, \Psi^{d}(U)$ are all disjoint. 

\begin{lemma}\label{lemma:eta-d-displacement_v1} Let $F$ be a Hamiltonian and let $B$ be an open topological disc which is $d$-displaced by $\varphi^1_F$. Then, for any Hamiltonian $G$ which is supported in $B$, we have $c_{d,k}(G \# F)=c_{d,k}(F)$. 
\end{lemma}

A similar lemma was established %
in \cite[Lemma 4.4]{CGHS} but only for maps supported in the northern hemisphere. The argument presented here is essentially the same and so we will be rather brief.

\begin{proof}[Proof of Lemma \ref{lemma:eta-d-displacement_v1}]
  Let $(K^s)_{[0,1]}$ be a smooth one parameter family of Hamiltonians such that for any $s\in[0,1]$, the time-one map of $K^s$ is $\varphi_G^s \varphi_F^1$ and such that the isotopy of $K^s$ consists in following first the isotopy generated by $F$ and then that generated by $sG(st,x)$. More precisely, we may take  $K^s = G^s \# F,$ where $G^s(x,t) := s G(x,st)$.  
It generates the isotopy
\[\varphi_{K^s}^t=
  \begin{cases}
\varphi_F^{\rho(2t)},&\quadif t\in[0,\tfrac12],\\
\varphi_G^{s\rho(2t-1)}\varphi^1_F   ,&\quadif t\in[\tfrac12, 1]. 
  \end{cases}
\]
  
Then, for all $s\in[0,1]$, $\spec_d(K^s)=\spec_d(F)$: the argument for this is exactly the same as the argument\footnote{To orient a reader who reads \cite[Lem. 4.4]{CGHS}, note that the argument there refers to the spectrum of the time $1$ maps of $K^s$ and $F$ rather than to the Hamiltonians themselves; this is because in that proof, the Hamiltonians are all assumed zero on the southern Hemisphere so we can refer to the spectrum in terms of the time $1$-map; however the argument for that Lemma extends to the case here with no changes.} in \cite[Lem. 4.4]{CGHS} and so we will omit it.  This implies that for any $(d,k)$ the continuous map $s\mapsto c_{d,k}(K^s)$ take values in $\spec_d(F)$; the fact that this map is continuous is a consequence of Hofer continuity of $c_{d,k}$, see \eqref{eqn:hofcont}. Since this set is totally discontinuous, we 
deduce that these functions are all constant, and so $c_{d,k}(K^0) = c_{d,k}(K^1)$.

Since $K^1=G\# F$, it is sufficient to show that $c_{d,k}(K^0) =  c_{d,k}(F)$ to finish the proof. 
To see this, note that the Hamiltonians flows $\varphi^t_{K^0}$ and $\varphi^t_{F}$ are homotopic rel.\ endpoints and, moreover,  $\int_{\S^1} \int_{\S^2} K^0 \omega \, dt =  \int_{\S^1} \int_{\S^2} F \omega \, dt$.  It then follows from the Homotopy invariance and Shift properties that $c_{d,k}(K^0) =  c_{d,k}(F)$.
\end{proof}

We will now use Lemma \ref{lemma:eta-d-displacement_v1} to establish the Support-Control inequality. 
\begin{proof}[Proof of the Support-control inequality] Fix $d>0$. Let $H$ be a Hamiltonian supported in a disc $D$ of area smaller than $\frac1{d+1}$. This area condition implies that we can find a Hamiltonian $F$ such that the disc $D$ is $d$-displaced by $\varphi^1_F$.  %

Furthermore, for any $\eps>0$, we may assume that $\Vert F \Vert_{1, \infty} \leq  \mathrm{Area}(D)+\eps$. To see this, note that we can find an area preserving diffeomorphism $\psi$ such that $\psi(D)$ is sandwiched between two meridians (that is, curves with $\theta = $ constant) of $\S^2$ which enclose a region of area $\mathrm{Area(D)}+\frac\eps2$. Suppose that $	\eps$ is so small that ${\mathrm{Area(D)}+\eps}<\frac1{d+1}$.  Then, consider the Hamiltonian $K = \frac{\mathrm{Area(D)}+\eps}2z$ whose time-1 map $\varphi^1_K$ is the horizontal rotation of angle $\mathrm{Area(D)}+\eps$, which $d$-displaces $\psi(D)$.  Then, we may set $F = K \circ \psi$, whose time-1 map, $\psi^{-1} \varphi^1_K \psi$, d-displaces the disc $D$. Clearly, $\Vert F \Vert_{1, \infty} = {\mathrm{Area(D)}+\eps}$.

By Lemma \ref{lemma:eta-d-displacement_v1}, we have $c_{d,k}(H \# F)=c_{d,k}(F)$.  Using this, and Equation \ref{eqn:hofcont}, we obtain
$$ | c_{d,k}(H) - c_{d,k}(F) | = | c_{d,k}(H\# 0) - c_{d,k}(H \# F) |  \leq d \Vert H\# 0 - H \# F \Vert_{1, \infty} = d \Vert F \Vert_{1, \infty}. $$
Hence, we have $$| c_{d,k}(H)| \leq |c_{d,k}(F)| + d \Vert F \Vert_{1, \infty} \leq 2d \Vert F \Vert_{1, \infty} = 2d \, \mathrm{Area(D)}+2\eps.$$
This completes the proof of support-control inequality.
\end{proof}

We have now completed the proof of Proposition \ref{prop:more}.
\end{proof}

 \subsection{Invariants for Hamiltonian diffeomorphisms}\label{sec:eta_d}

The goal of this section is to introduce PFH spectral invariants, and other related invariants, for Hamiltonian diffeomorphisms.

 \subsubsection{The $c_d$}
We begin by noting that we can now define the PFH spectral invariants on the universal cover $\widetilde{\Ham}(\S^2, \omega)$ as follows.  Given $\tilde{\varphi} \in \widetilde{\Ham}(\S^2, \omega)$, let $H$ be a mean-normalized Hamiltonian such that the Hamiltonian path $\{\varphi^t_H\}, 0 \le t \le 1$, is  a representative for $\tilde \varphi$.  Define
 \begin{equation}\label{eq:c_d_universal_cover}
  c_{d,k}(\tilde{\varphi}) := c_{d,k}(H).
 \end{equation}
 
 The mapping $$c_{d,k}: \widetilde{\Ham}(\S^2, \omega) \rightarrow \R$$ is well-defined as a consequence of the Homotopy invariance property in Proposition \ref{prop:more}.
Note that for any (not necessarily normalized) Hamiltonian $H$, the Shift property yields 
\begin{equation}
c_{d,k}(\tilde\varphi)=c_{d,k}(H)-d\int_{\S^1}\int_{\S^2}H_t\,\omega\,dt,\label{eq:relation-cd_cd-d}
\end{equation}
where $\tilde\varphi$ is the lift of $\varphi$ given by the isotopy $(\varphi_H^t)_{t\in[0,1]}$.

However, to prove our main theorems, we will want invariants of $\Ham(\S^2,\omega)$ rather than $\widetilde{\Ham}(\S^2,\omega)$.

To produce such invariants, we start by showing that, as mentioned above, it turns out that we can use the $c_{d,k}$ to define invariants that are independent of the choice of mean normalized Hamiltonian.  To get started with this, let $\tilde\varphi \in \widetilde{\Ham}(\S^2,\omega).$ Define
\[ c_d(\tilde\varphi) := c_{d,-d}(\tilde\varphi).\]

Next, we will prove that, for even $d$, the map $c_d: \widetilde{\Ham}(\S^2,\omega)  \rightarrow \R$ descends to $\Ham(\S^2, \omega)$.  In other words, we will show that there is a well-defined map 
\[ c_d: \Ham(\S^2, \omega) \rightarrow \R.\]
This is the content of Proposition~\ref{prop:cdeven}.

\begin{remark}
\label{rmk:noncanonical}
The $c_d$ as defined here are not canonical.  We could equally well define
\[ c_d(\varphi) := c_{d,k}(\varphi_H^1)\]
for any $-d \le k \le d$ with the same party as $d$.  What is important for the applications in our paper is to choose a $k$ such that the $c_{d,k}$ satisfy the Normalization property.  It is also instructive to note that because addition of 
the homology class of a fiber of the map $Y_{\varphi} \to S^1$ 
induces a canonical bijection of the twisted PFH chain complex, we have
\begin{equation}
\label{eqn:sphere}
c_{d,k+2d+2}(H) = c_{d,k}(H) + 1.
\end{equation}  
In particular, as a function on $C^{\infty}(\mathbb{S}^1 \times \mathbb{S}^2)$, any $c_{d,k'}$ differs by a constant function from some $c_{d,k}$ with $-d \le k \le d$.

To summarize, then, there are essentially $d+1$ possible spectral invariants corresponding to degree $d$, and we have made a non-canonical choice of one of them moving forward, with the main goal of simplifying the notation.
\end{remark}

For future use, we also define
\[c_d(H):=c_{d,-d}(H), \]
for any $H\in C^\infty(\S^1 \times \S^2)$.

\begin{prop}
\label{prop:cdeven} 
For any positive even integer $d$ and any even integer $k$, the invariant $c_{d,k}: \widetilde{\Ham}(\S^2, \omega) \rightarrow \R$ descends to $\Ham(\S^2, \omega)$.  In other words, it does not depend on the choice of mean-normalized $H$.
\end{prop}

In particular, we obtain a well-defined invariant $c_d:\Ham(\S^2,\omega)\to\R$ for any positive even integer $d$.

\begin{proof} Let $H$ be any Hamiltonian and $K = \frac{1}{2} (z+1).$  Note that the Hamiltonian $K$ vanishes at $p_-$ and its time $1$ flow is rotation by $2 \pi$ about the $z$-axis.   We will show below that for any positive integer $d$ and integer $k$ of the same parity as $d$, 
\begin{equation}
\label{eqn:rotcomp}
c_{d, k}(H \# K) = c_{d,k}(H) + \tfrac{d}{2}.
\end{equation}   
This implies Proposition~\ref{prop:cdeven}, by the following argument.  Let $H_1$ and $H_2$ be mean-normalized Hamiltonians generating the same time $1$-map.  We can assume that $H_1$ and $H_2$ are not homotopic,
or else the proposition holds by Proposition \ref{prop:more}, item 5.  Then, $H_1 \# (K - \frac{1}{2})$ and $H_2$ are homotopic, and $H_1 \# (K - \tfrac{1}{2})$ is mean-normalized, hence
\[ c_{d, k}(H_1 \# (K - \tfrac{1}{2}) ) = c_{d, k}(H_2).\] 
On the other hand, by the Shift property       
\[ c_{d, k}(H_1 \# (K - \tfrac{1}{2}) ) = c_{d, k}(H_1 \# K) - \frac{d}{2},\]
so that the Proposition follows from \eqref{eqn:rotcomp}.  

It remains to prove \eqref{eqn:rotcomp}.  To prove this, we first note that
\[ c_{d, k}(H \# 0) = c_{d, k}(H).\]
Indeed, $H \# 0$ and $H$ are homotopic, with the same mean.  Thus, it suffices to show that
\[ c_{d, k}(H \# K) = c_{d, k}(H \# 0) + \frac{d}{2}.\]
To prove this, by the Shift property we can assume that $H$ vanishes at $p_-$.  Then, $H \# K$ and $H \# 0$ have the same time $1$-map $\varphi$, and the same reference cycle $\gamma \subset Y_{\varphi}$.  The only difference between them is the trivialization of $V$ over $\gamma$; more precisely, if $\tau'$ denotes the trivialization over $\gamma$ induced by $H \# K$ and $\tau$ denotes the trivialization induced by $H \# 0$, then we have $\tau' = \tau - 1$.  Thus, since in this case the identity map is an isomorphism of the twisted PFH chain complexes, which shifts the grading by $d^2 + d$ by \cite[Eq. 6, Lem. 2.5.b]{Hutchings-index}, 
we have
\[ c_{d,k}(H \# K) = c_{d,k+d^2+d}(H \# 0) = c_{d,k}(H \# 0) + \frac{d}{2}\]
as desired; here, we have used \eqref{eqn:sphere} for the second equality above. 
\end{proof}

\subsubsection{Homogenized invariants}

We introduced the homogenizations $\mu_d$ and $\zeta_d$ in the introduction (see Equations (\ref{eq:mu_d}) and (\ref{eqn:zeta_d_def}). The next proposition states that they are well-defined.

\begin{prop}\label{prop:mu_d-well-defined}
  There are well defined maps $\mu_d:\Ham(\S^2,\omega)\to\R$ and $\zeta_d:C^\infty(\S^2)\to\R$ given by
  \[\mu_d(\varphi)=\limsup_{n\to\infty}\frac{c_d(\tilde\varphi^n)}n\quadandquad \zeta_d(H)=\limsup_{n\to\infty}\frac{c_{d}(nH)}n,\]
  for any  $\varphi\in\Ham(\S^2,\omega)$ and $H\in C^\infty(\S^2)$.
\end{prop}

\begin{remark} One can more generally define $\zeta_d(H):=\limsup_{n\to\infty}\frac{c_{d}(H^n)}n$ for any (non necessarily autonomous) Hamiltonian $H\in C^\infty(\S^1\times\S^2)$. However, we choose to restrict $\zeta_d$ to $C^\infty(\S^2)$ in analogy with \cite{Entov-Polterovich-Quasi-States}, where a similar map $\zeta$ was defined and was proved to satisfy the properties of a symplectic quasi-state. It would be interesting to see if our $\zeta_d$ also has these properties.
\end{remark}

\begin{proof} The fact that both of the above $\limsup$ exist follows directly from the Continuity property of $c_d:=c_{d,-d}$ in  Theorem \ref{thm:PFHspec_initial_properties}.  This shows that $\zeta_d$ is well defined on $C^\infty(\S^2)$ and that $\mu_d$ is well-defined on the universal cover $\Hamtilde(\S^2,\omega)$. It remains to show that $\mu_d$ descends to $\Ham(\S^2,\omega)$.

  To see this, let $\varphi\in\Ham(\S^2,\omega)$. Let $\tilde\varphi, \tilde\varphi'$ be two lifts of $\varphi$ to $\widetilde{\Ham}(\S^2,\omega)$. Recall that $\pi_1(\Ham(\S^2,\omega)$ has only two elements and that the non-trivial element is represented by the isotopy $\{\varphi_R^t\}$, where $R(\theta,z)=\frac z2$ is the Hamiltonian which generates a full $2\pi$ rotation around the $z$-axis. As a consequence, for any $n\in\N$, since $\tilde\varphi^n$ and $\tilde\varphi'^n$ are both lifts of the same diffeomorphism $\varphi^n$,  we have either $\tilde\varphi^n=\tilde\varphi'^n$ or  $\tilde\varphi^n=\tilde\varphi'^n\tilde\varphi_R$.
In both cases, the Continuity property of Theorem \ref{thm:PFHspec_initial_properties}
gives an upper bound which does not depend on $n$:
  \[\left|c_d(\tilde\varphi^n)-c_d(\tilde\varphi'^n)\right|\leq d\,\|R\|_{1,\infty}.\]
  It then follows that
  \[\limsup_{n\to\infty}\frac{c_d(\tilde\varphi^n)}n=\limsup_{n\to\infty}\frac{c_d(\tilde\varphi'^n)}n,\]
  which proves that $\mu_d$ descends to $\Ham(\S^2,\omega)$.
\end{proof}

We next state some of the properties of $\mu_d$ which will be used in our arguments.

\begin{prop}\label{prop:mu_d_properties}
The invariant $\mu_d : \Ham(\S^2, \omega) \rightarrow \R$ satisfies the following properties:
 \begin{enumerate}
 \item Normalization:  $\mu_d(\id) =0$. 
 \item Hofer continuity: For all $\varphi, \psi$ we have
   \[ |\mu_d(\varphi) - \mu_d(\psi) | \leq d\, d_H(\varphi, \psi). \] %
  \item Calabi property: Suppose that $\supp(\varphi)$ is contained in a topological disc $D$.  If $\area(D) < \frac{1}{d+1} $, then $$ \frac1d\mu_d(\varphi) = -\Cal(\varphi),$$
where $\Cal:\Ham_c(D,\omega)\to\R$ denotes the Calabi invariant.
\item Relationship with $\zeta_d$: For any $H \in C^{\infty}(\S^2)$, 
   \begin{equation}
   \label{eqn:mudzeta}
   \mu_d(\varphi^t_H)= \zeta_d(tH) -  t\,d \int_{\S^2} H \omega,
   \end{equation}
  for all $t\in\R$. 
\end{enumerate}
  \end{prop}
  
 \begin{proof}
 The first item follows immediately from the definition of $\mu_d$ combined with the fact that  $c_{d}(0) =0$. 

To prove the second item, let $\varphi, \psi\in\Ham(\S^2,\omega)$ and $H,K$ be mean-normalized Hamiltonians with $\varphi_H^1=\varphi$ and $\varphi_K^1=\psi$. We also denote by $\tilde\varphi, \tilde\psi\in\Hamtilde(\S^2,\omega)$ the lifts of $\varphi,\psi$ respectively given by $H,K$.
Then, by definition 
\[\tfrac1d|c_d(\tilde\varphi^n) - c_d(\tilde\psi^n)|=\tfrac1d|c_{d}(H^{\# n})-c_{d}(K^{\# n})|,\]
for any $n>0$. Here $H^{\# n}$ denotes the $n$-fold composition 
$H\#\dots\# H$.
By the Continuity property of $c_{d}$, we have
\[\tfrac1d\,|c_{d}(H^n)-c_{d}(K^n)| \leq  \|H^{\# n}-K^{\# n}\|_{1,\infty}= n\|H-K\|_{1,\infty}.\]
Note that this last equality follows from (\ref{eq:norm-and-join}).
From this inequality, we deduce 
\[|\mu_d(\varphi) - \mu_d(\psi)|\leq d \|H-K\|_{1,\infty}.\]
Since this holds for any choices of Hamiltonians $H,K$,  and since we can restrict to mean-normalized Hamiltonians in computing the Hofer norm of $\varphi^{-1}\psi$, the Hofer continuity property follows.

The Calabi property is a consequence of the Support-control property from Prop. \ref{prop:more}.
Indeed, given any Hamiltonian $H$ with support in $D$, 
\eqref{eq:relation-cd_cd-d} yields 
\begin{equation*}
c_d(\tilde\varphi)=c_{d}(H)-d \int_{\S^1}\int_{\S^2}H_t\,\omega\,dt,
\end{equation*}
where $\tilde\varphi$ is the lift of $\varphi$ given by the isotopy $(\varphi_H^t)_{t\in[0,1]}$. The integral in the second term in the right hand side above is nothing but $-\Cal(\varphi)$, while the first term is bounded from above by $2d\,\mathrm{Area}(D)$, by Support-control. Applying this to $\varphi^n$, for any $n>0$, we get
\[|\tfrac1n c_d(\tilde\varphi^n)+d\,\Cal(\varphi^n)| = \tfrac1n|c_d(\tilde\varphi^n)+d\,\Cal(\varphi^n)|=\tfrac1n|c_{d}(H^n)|\leq \tfrac1n 2d\,\mathrm{Area}(D).\]
 The Calabi property follows from this last inequality.

As for the last item, it follows from the definitions of $\mu_d$ and $\zeta_d$, and 
 \eqref{eq:relation-cd_cd-d} that
 $$\mu_d(\varphi^t_H) = \mu_d(\varphi^1_{tH}) = \zeta_d(tH) - d\int_{\S^1}\int_{\S^2}tH\,\omega\,dt.$$ 
 \end{proof}

\subsubsection{The invariants $\eta_d$}

 Although we can use the invariants $c_d$ or $\mu_d$ to get invariants of the time-$1$ map, these invariants will not in general be $C^0$-continuous, as they require mean normalizing the Hamiltonian.  
 We obtain $C^0$-continuous invariants by defining, for even $d\in \N$, the numbers
\begin{equation} \label{eq:eta_d}
\begin{split}
 \eta_d : \;    \Ham(\S^2, \omega) \rightarrow \R \\
  \varphi \mapsto c_d(\varphi) - \frac{d}{2} c_2(\varphi).
\end{split}
\end{equation}

The fact that $\eta_d$ is well-defined is an immediate consequence of Proposition \ref{prop:cdeven}. Observe that, by Proposition \ref{prop:cdeven} and the Shift property of Proposition \ref{prop:more}, we have 
\begin{equation}\label{eqn:eta_d2}
\eta_d(\varphi) = c_d(H) - \frac{d}{2}c_2(H),
\end{equation}
where $H$ is any Hamiltonian with time-$1$ flow $\varphi$.

 We now prove that $\eta_d$ satisfies various properties, most notably $C^0$--continuity.

\begin{prop}\label{prop:eta_d_properties}
The invariant $\eta_d$ is well-defined and satisfies the following properties for all $\varphi, \psi \in \Ham(\S^2, \omega)$.
\begin{enumerate}
\item Normalization: $\eta_d(\id) = 0$.
\item Hofer continuity:   
 $ \vert  \eta_d(\varphi) - \eta_d(\psi) \vert  \leq  d \; d_H( \varphi, \psi) .$
\item Support-control: If the support of $\varphi \in \Ham(\S^2, \omega)$, is contained in a topological disc $D$ with $\area(D) < \frac{1}{d+1} $, then $\vert \eta_d(\varphi)|\leq 2d \; \area(D)$.
\item $C^0$ continuity: The mapping $\eta_d: \Ham(\S^2, \omega)\rightarrow \R$ is continuous with respect to the $C^0$ topology on $\Ham(\S^2, \omega)$ and, moreover, it extends continuously to $\Homeo_0(\S^2, \omega)$.
\end{enumerate}
\end{prop}

We recall in relation to the support-control inequality that the total area of the sphere is assumed to be one.

\begin{proof}[Proof of Proposition~\ref{prop:eta_d_properties}]
The first and third properties are immediate consequences of the same properties of the invariants $c_{d,k}$; see  Theorem \ref{thm:PFHspec_initial_properties} and Proposition \ref{prop:more}.  
To prove the second, note first of all that if $H$ and $G$ are any Hamiltonians, then by the Continuity property of Theorem \ref{thm:PFHspec_initial_properties}, we have
\[ (c_d(H) - c_d(G)) - \frac{d}{2}\left( c_2(H) - c_2(G)\right) \le d || H - G ||.\]
Now let $K$ be any Hamiltonian generating $\varphi^{-1}\psi$ and let $G$ generate $\varphi$.  Then $H := G \# K$ generates $\psi$ and hence by the above inequality    
\[ \eta_d(\psi) - \eta_d(\varphi) \le d || K ||, \quad \eta_d(\varphi) - \eta_d(\psi) \le d || K ||, \]
hence the Hofer continuity property, since $K$ was arbitrary.

We only have to establish the $C^0$-Continuity property.  This takes up the remainder of this subsection. Our proof will follow the lines (and use some of the intermediate steps) of \cite[Section 4]{CGHS}, which established a similar result for the invariant $c_{d}$ restricted to maps supported in the northern hemisphere. We fix some degree $d>0$. The result will follow from the next proposition.

\begin{prop}\label{prop:C0-continuity-elsewhere} For any $h\in\Homeo_0(\S^2,\omega)$ and $\eps>0$, there exists  $\delta>0$ such that for all $f,g\in \Ham(\S^2, \omega)$ satisfying $d_{C^0}(f, h)<\delta$ and $d_{C^0}(g,\id)<\delta$,  the inequality $|\eta_d(gf)-\eta_d(f)|<\eps$ holds. 
\end{prop}

Let us temporarily assume this proposition and explain how it implies the $C^0$-Continuity property of $\eta_d$. Let $(f_i)_{i\in\N}\in\Ham(\S^2,\omega)$ be a sequence which $C^0$ converges to $h\in\Ham(\S^2,\omega)$. We may write $f_i$ in the form $g_ih$, with $d_{C^0}(g_i,\id)\to 0$ as $i$ goes to $\infty$. By the proposition we have $|\eta_d(g_ih)-\eta_d(h)|\to 0$, which proves the $C^0$-continuity. To prove extension to $\Homeo_0(\S^2, \omega)$ let $h\in\Homeo_0(\S^2,\omega)$ and let $(f_i)_{i\in\N}\in\Ham(\S^2, \omega)$ be a sequence which $C^0$-converges to $h$. Then, $d_{C^0}(f_i,f_j)=d_{C^0}(f_if_j\inv, \id)$ becomes arbitrarily small when $i,j$ are large enough and so Proposition \ref{prop:C0-continuity-elsewhere} implies that 
$|\eta_d(f_i)-\eta_d(f_j)|=|\eta_d((f_if_j\inv)f_j)-\eta_d(f_j)|$ becomes arbitrarily small for $i,j$ large enough so that 
$\eta_d(f_i)$ converges. 
Proposition \ref{prop:C0-continuity-elsewhere} also similarly implies that if $(f_i')_{i\in\N}$ is another sequence converging to $h$, then $|\eta_d(f_i)-\eta_d(f_i')|\to 0$, hence the limit does not depend on the choice of limiting sequence. This allows us to consistently define $\eta_d(h)$ for any $h\in\Homeo_0(\S^2,\omega)$ by setting
\[\eta_d(h):=\lim_{i\to\infty}\eta_d(f_i),\]
for any sequence $f_i$ which $C^0$-converges to $h$.
\end{proof}

We now prove Proposition \ref{prop:C0-continuity-elsewhere}.

\begin{proof}[Proof of Proposition \ref{prop:C0-continuity-elsewhere}]
We give the proof of this proposition in two steps.
 
 {\em Step A.  Continuity in the non-finite order case.}

 We first assume that $h$ is not of finite order in the group $\Homeo_0(\S^2, \omega)$.  Then, there exists a point $x\in\S^2$ such that $h^{d!}(x)\neq x$.  For such a point and for any integers $0\leq p<q\leq d$, we have $h^{q-p}(x)\neq x$. By composing with $h^p$, we also have $h^q(x)\neq h^p(x)$. Therefore, the points $x, h(x), \dots, h^d(x)$ are pairwise distinct. Let $B$ be a small ball centered at $x$, such that the closure $\overline{B}$ of $B$ is $d$-displaced by $h$. 

Let $\eps>0$.  We choose $\delta'>0$ so small that any map $f$ such that $d_{C^0}(f,h)<\delta'$ must also $d$-displace  $\overline{B}$. 

The next lemma says roughly that a $C^0$-small element of $\Ham(\S^2,\omega)$ is Hofer-close to being supported in $B$. 

\begin{lemma}\label{lemma:hofer-close-to-supported-in-ball} Let $B$ be any open topological disc. For all $\eps'>0$, there exists $\delta>0$, such that for all $g\in\Ham(\S^2, \omega)$ with $d_{C^0}(g,\id)<\delta$, there is $\phi\in\Ham(\S^2, \omega)$ supported in $B$ such that $d_H(\phi,g)\leq\eps'$.
\end{lemma}

A similar result was proved in \cite[Lemma 4.6]{CGHS} but only for maps $g$ supported in the northern hemisphere. By conjugating by an appropriate area preserving map, this particular case implies that Lemma \ref{lemma:hofer-close-to-supported-in-ball} holds for maps $g$ supported in any topological disk of area $\frac12$.  In fact, the factor $\frac{1}{2}$ here is not essential to the proof of \cite[Lemma 4.6]{CGHS}: the exact same argument, which we omit\footnote{To help the reader who reads the argument in \cite[Lemma 4.6]{CGHS}, we note that the only change is that the factors of $1/2$, which come from the fact that the northern hemisphere has area $1/2$, see the end of the second paragraph of the proof there, must be changed to some number $\ell < 1$.  This change can be accommodated by choosing what are called $N$ and $m$ in the proof to be such that $\ell/N < \text{area}(B), \ell/m < 1/2$ and $4 \ell \frac{N+1}{m}  < \epsilon.$  With the preceding changes understood, the argument can then be repeated verbatim.}     
for brevity, shows
that it also holds for maps supported in an embedded disk of any area 
The general case then immediately follows from the next fragmentation lemma: indeed, given the lemma below, and given $g$, we can first fragment $g$ into maps supported on embedded disks and then approximate each of these maps by maps supported in $B$.  

\begin{lemma}\cite[Prop 3.1]{Sey13}\label{lemma:C0-frag} There exists two open topological embedded discs $D_1,D_2$ which cover $\S^2$, such that for any $\alpha>0$, there exists $\delta>0$ such that for any $g\in\Ham(S^2, \omega)$ satisfying $d_{C^0}(g, \id)<\delta$, there exists $g_1,g_2\in\Ham(\S^2, \omega)$, with $\supp(g_i)\subset D_i$ and $d_{C^0}(g_i,\id)< \alpha$ for $i=1,2$, such that $g=g_1\circ g_2$.
\end{lemma}

Having established Lemma \ref{lemma:hofer-close-to-supported-in-ball}, we can now continue the proof of Proposition \ref{prop:C0-continuity-elsewhere}.
Let $\delta>0$ be as provided by 
Lemma \ref{lemma:hofer-close-to-supported-in-ball} for $\eps'=\frac{\eps}d$.
 We may assume without loss of generality that $\delta<\delta'$.
Then, by Lemma \ref{lemma:hofer-close-to-supported-in-ball}, for any Hamiltonian diffeomorphism $g$ with $d_{C^0}(g,\id)<\delta$, there exists $\phi\in\Ham(\S^2, \omega)$ supported in $B$ satisfying $d_H(\phi, g)\leq \frac{\eps}d$. Now let $f, g\in \Ham(\S^2, \omega)$ be such that $d_{C^0}(f,h)<\delta$ and $d_{C^0}(g,\id)<\delta$. By Hofer continuity of $\eta_d$ and Lemma \ref{lemma:eta-d-displacement_v1}, 
it follows that
\begin{align*}
  |\eta_d(gf)-\eta_d(f)|&\leq |\eta_d(gf)-\eta_d(\phi f)| + |\eta_d(\phi f)-\eta_d(f)|\\
                        & \leq d\, d_H(gf, \phi f) + 0= d\,d_H(g,\phi)\leq \eps.
\end{align*}
This concludes the proof in the case where $h$ is not of finite order.

\medskip

{\em Step B.  The finite-order case.}  

We will now conclude the argument by reducing to the case where $h$ is of finite order to the case where it is not. Let $h$ be of finite order and let $\eps>0$. We may pick a Hamiltonian diffeomorphism $\psi$ such that $\|\psi\|<\frac\eps3$ and $h\psi$ is not of finite order\footnote{For the benefit of the reader, we briefly sketch why such a $\psi$ exists. Since $h$ has finite order, all $x\in\S^2$ are periodic and we let $\ell$ be the maximal period of a point. Then, the set of points of period $\ell$ is open, because it is $\{x\in\S^2: h^k(x)\neq x, \forall k=1,\dots, \ell-1\}$. It follows that if we fix a point $x$ of period $\ell$, there exists an open set $U$ containing $x$ such that $h^\ell|_U=\id_U$ and $U, \dots, h^{\ell-1}(U)$ are all disjoint. Now, let $\psi$ be a $C^1$-small (hence Hofer small) map supported in $U$ which coincides with an irrational rotation around $x$ in a smaller open subset $V\subset U$. Then, $h\psi$ does not have finite order. Indeed, for any $n\in\N$, $y\in V\setminus \{x\}$, we have $(h\psi)^n(y)=\psi^{n/\ell}(y)$ if $\ell$ divides $n$ and $(h\psi)^n(y)\notin U$ otherwise. In both cases, $(h\psi)^n(y)\neq y$. Thus, such $\psi$ suits our needs.}
 Then, by Step A, there exists $\delta$ such that for any $f', g'$ satisfying $d_{C^0}(f',h\psi)<\delta$ and $d_{C^0}(g', h)<\delta$, we have $|\eta_d(g'f')-\eta_d(f')|<\frac\eps3$.

Now take $f, g$ as in the statement of the proposition.  We now apply the triangle inequality to obtain 
\[ |\eta_d(gf) - \eta_d(f)| \le | \eta_d(gf) - \eta_d(gf\psi)| + |\eta_d(gf\psi) - \eta_d(f\psi)| + |\eta_d(f \psi) - \eta_d(f) | .\]
By Hofer continuity and the above estimate on $\|\psi\|$, we have
\[ | \eta_d(gf) - \eta_d(gf\psi)| \le \frac{\eps}{3}, \quad | \eta_d(f\psi) - \eta_d(f)| \le \frac{\eps}{3}.\]
Thus, to finish the proof of the proposition, it remains to show that
\[  |\eta_d(gf\psi) - \eta_d(f\psi)| \le \frac{\eps}{3}.\]
This follows from the previous paragraph, since if $d_{C^0}(f,h)<\delta,$ then $d_{C^0}(f\psi, h\psi)<\delta$.
\end{proof}

\section{The quasi-isometry type of the kernel of Calabi}
\label{sec:kappol}

Equipped with our new spectral invariants, we now prove Theorem~\ref{theo:QI-ker-Cal}.

\subsection{Homogenization and monotone twists}

\subsubsection{A combinatorial model} 
\label{sec:comb}

We begin by recalling the combinatorial model of Theorem 6.1 of \cite{CGHS}, which gives an explicit formula for $c_d(H)$ where $H:\S^2 \rightarrow \R$ is a monotone twist, and which we will need below and for the proof of Theorem~\ref{theo:non-simplicity} as well.     

 Here and below we use the notations of \cite[Section 5.2]{CGHS}.  To summarize, recall that a {\bf lattice path} $P$ is the graph of a piecewise linear function $Y: [0,d] \rightarrow  \mathbb{R}_{\ge 0}$, such that the vertices of $P$ are at integer lattice points; the number $d$ is called the {\bf degree} of the path and is assumed an integer below.  A lattice path is called {\bf concave} if it never crosses below any of its tangent lines.  We can think of a lattice path as a collection of maximal line segments, called {\bf edges}, joined end to end.  We regard any edge as an integer multiple $m_{p,q}$ of a primitive vector $(q,p)$; these vectors are directed with the convention that $q \ge 1.$

To any concave lattice path, we associate a number $j(P)$ as follows.  We form the region $R_+$ bounded by the $x$-axis, the line $x = d$, and the part of $P$ above the $x$-axis, we form the region $R_-$ bounded by the axes and the part of $P$ below the $x$-axis, we define $j_+$ to be the number of lattice points in $R_+$, not including lattice points on $P$, and we define $j_-$ to be the number of lattice points in $R_-$, not including the lattice points on the $x$-axis.  Finally, we define $j(P) := j_+ - j_-$.  See Figure \ref{fig:combinatorial-index}.
We define the {\bf combinatorial index}  $I(P)$ by $$I(P) : = 2 j(P) -d .$$  

 \begin{figure}[h!]
 \centering 
 \def\svgwidth{0.5\textwidth} 
\begingroup%
  \makeatletter%
  \providecommand\color[2][]{%
    \errmessage{(Inkscape) Color is used for the text in Inkscape, but the package 'color.sty' is not loaded}%
    \renewcommand\color[2][]{}%
  }%
  \providecommand\transparent[1]{%
    \errmessage{(Inkscape) Transparency is used (non-zero) for the text in Inkscape, but the package 'transparent.sty' is not loaded}%
    \renewcommand\transparent[1]{}%
  }%
  \providecommand\rotatebox[2]{#2}%
  \newcommand*\fsize{\dimexpr\f@size pt\relax}%
  \newcommand*\lineheight[1]{\fontsize{\fsize}{#1\fsize}\selectfont}%
  \ifx\svgwidth\undefined%
    \setlength{\unitlength}{75.13666952bp}%
    \ifx\svgscale\undefined%
      \relax%
    \else%
      \setlength{\unitlength}{\unitlength * \real{\svgscale}}%
    \fi%
  \else%
    \setlength{\unitlength}{\svgwidth}%
  \fi%
  \global\let\svgwidth\undefined%
  \global\let\svgscale\undefined%
  \makeatother%
  \begin{picture}(1,0.69626761)%
    \lineheight{1}%
    \setlength\tabcolsep{0pt}%
    \put(0,0){\includegraphics[width=\unitlength,page=1]{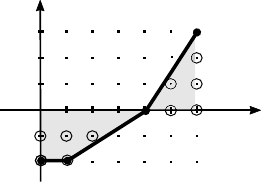}}%
    \put(0.60285862,0.57377551){\color[rgb]{0,0,0}\makebox(0,0)[lt]{\lineheight{1.25}\smash{\begin{tabular}[t]{l}$P$\end{tabular}}}}%
    \put(0.75850718,0.2900258){\color[rgb]{0,0,0}\makebox(0,0)[lt]{\lineheight{1.25}\smash{\begin{tabular}[t]{l}$d$\end{tabular}}}}%
    \put(0.10876709,0.29301464){\color[rgb]{0,0,0}\makebox(0,0)[lt]{\lineheight{1.25}\smash{\begin{tabular}[t]{l}$0$\end{tabular}}}}%
    \put(0.6675432,0.40162053){\color[rgb]{0,0,0}\makebox(0,0)[lt]{\lineheight{1.25}\smash{\begin{tabular}[t]{l}$R_+$\end{tabular}}}}%
    \put(0.26086351,0.2055008){\color[rgb]{0,0,0}\makebox(0,0)[lt]{\lineheight{1.25}\smash{\begin{tabular}[t]{l}$R_-$\end{tabular}}}}%
  \end{picture}%
\endgroup%

 \caption{The lattice points which contribute to the count for $j(P)$ are circled. Here, $j_+(P)=5$, $j_-(P)=5$, $d=6$, thus $j(P)=0$ and $I(P)=-6$.}\label{fig:combinatorial-index}
\end{figure}
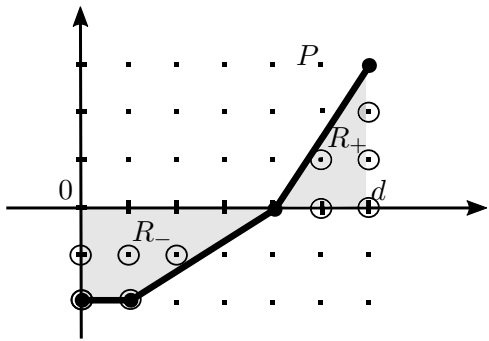

Now let $H = \frac{1}{2}h$ be a monotone twist. 
Assume that $h(-1) = h'(-1) = 0, h'' > 0$, and $h'(1)$ is an integer.  We call such a monotone twist {\bf nice}.  We call a lattice path $P$ {\bf compatible} with $h$ if for every edge $m_{p,q} (q,p)$, there exists some $z_{p,q}$ such that $h'(z_{p,q} ) = p/q.$  If $P$ is a concave lattice path, compatible with a nice monotone twist, then we define the {\bf action} $\mathcal{A}(P)$ by defining
\begin{equation}
 \mathcal{A}(q,p) := \frac{1}{2} \left( p (1-z_{p,q} ) + q h(z_{p,q})\right) ,\label{eq:action-lattice-path}
\end{equation}
and extending by linearity.

We can now state the formula from     
\cite[Thm. 6.1]{CGHS} for computing the invariants $c_{d,k}$.  That is, for all degree $d\geq 1$ and all grading $k$, when $H$ is a nice monotone twist we have 
 \begin{equation}
 \label{eqn:specform}
 c_{d,k}(H)=\max\{\cal A(P):2j(P)-d=k\},
 \end{equation}
where the maximum is over concave lattice paths that are compatible with $h$.  We are justified in invoking this formula because our nice monotone twists are in $\mathcal{H}$, and the $c_{d,k}$ defined here extend the definition from \cite{CGHS}, see the discussion at the beginning of Section~\ref{sec:newspec}.

\subsubsection{$\zeta_d$ of monotone twists}

We now apply the combinatorial model from the previous section to compute the invariants $\zeta_d$ in the case of monotone twists.
In particular, we can now give the promised proof of Proposition~\ref{prop:zeta-d}.

\begin{proof}[Proof of Proposition~\ref{prop:zeta-d}]
  First note that  $\zeta_d$ is Lipschitz continuous with respect to the uniform norm on $C^\infty(\S^2)$; this follows from Hofer continuity of $c_d$. Thus, both sides of the equation in Proposition \ref{prop:zeta-d} are continuous with respect to uniform norm. Since any monotone twist $H$ can be approximated uniformly by nice monotone twists, we deduce that it is sufficient to prove Proposition \ref{prop:zeta-d} for nice monotone twists. For the rest of the proof, we therefore assume that $H$ is a nice monotone twist. 
  
 We note that the index $I(P)  = 2 j(P) - d$ from \ref{sec:comb}
 is equivalently given by the following formula
\begin{equation}
I(P)=2A+y+w-e,\label{eq:index-aire}
\end{equation}
where $y$ and $w$ are respectively the minimal and maximal vertical coordinate of $P$, $e$ is the number of edges in $P$ and $A$ is the (signed) area of the region enclosed by $P$, the $x$-axis and the vertical line $\{d\}\times\R$. Indeed, by shifting the path if necessary, it suffices 
to prove this when $y = 0$, in which case it follows from Pick's formula that this corresponds to the definition given in \ref{sec:comb}.

 Let us introduce some notation. For any $i\in\{1,\dots, d\}$, we set $a_i=Y(i)-Y(i-1)$, where $Y$ is the function $[0,d]\to\R$, such that $P=\mathrm{graph}(Y)=\{(x,Y(x)):x\in[0,d]\}$. Then,
 \begin{align*}
   2A&= 2dy + a_1 +(2a_1 +a_2)+\dots +(2a_1+\dots+2a_{d-1}+a_d)\\
   &=2dy+ (2d-1)a_1+(2d-3)a_2+\dots+a_d
 \end{align*}
 Using (\ref{eq:index-aire}) and the relation $w=y+a_1+\dots+a_d$, the condition $I(P)=-d$ becomes
\[2dy+2y+2da_1+(2d-2)a_2+\dots+2a_d-e=-d.\]
Therefore, under this condition, we may express $y$ in terms of the $a_i$ as:
  \begin{equation}
   \label{eq:formula-for-y}
y=-\tfrac{1}{d+1}(da_1+(d-1)a_2+\dots a_d+\tfrac{d-e}2)=\tfrac{e-d}{2(d+1)}+ \sum^d_{i=1} (- a_i+\tfrac{i}{d+1}a_i).   
\end{equation}

Let us now turn our attention to the action. It is given by (\ref{eq:action-lattice-path}): 
we state here a reformulated version %
 with the $a_i$, namely
 \[\cal A(P)=y+\sum^d_{i=1}\tfrac12(a_i(1-z_i)+h(z_i)),\]
where $z_i$ is the unique point such that $h'(z_i)=a_i$. 
Using (\ref{eq:formula-for-y}), we obtain 
\begin{align}
  \cal A(P)&=\tfrac{e-d}{2(d+1)} + \sum^d_{i=1} (\tfrac{i}{d+1}a_i - a_i)+\sum^d_{i=1}\tfrac12(a_i(1-z_i)+h(z_i)),\nonumber\\
  &=\tfrac{e-d}{2(d+1)}+\tfrac12\sum_{i=1}^{d}\left(h'(z_i)(-1+\tfrac{2i}{d+1}-z_i)+h(z_i)\right)\label{eq:action-comb}
\end{align}
Note that the term $\tfrac{e-d}{2(d+1)}$ belongs to $(-\tfrac12, 0]$. It vanishes when all edges in $P$ have horizontal displacement $1$. 

By \eqref{eqn:specform}, 
$c_d = c_{d,-d}$ is obtained 
by maximizing the value of $\cal A(P)$ over all possible paths with $I(P)=-d$.

To compute this maximum, 
consider the function 
\[ F(t_1, \dots, t_d)=\sum_{i=1}^{d}\left(h'(t_i)(-1+\tfrac{2i}{d+1}-t_i)+h(t_i)\right),\] 
defined on the set $E$ of tuples $(t_1, \dots, t_d)$ such that $-1\leq t_1\leq\dots\leq t_d\leq 1$. We have 
\begin{equation}
\label{eqn:A}
\cal A(P)=\tfrac12 F(z_1, \dots, z_d)+\tfrac{e-d}{2(d+1)}.
\end{equation}
We may compute the partial derivatives of $F$,
\[\frac{\partial F}{\partial t_i}=(-1+\tfrac{2i}{d+1} - t_i)h''(t_i),\]
and we see that it is positive for $t_i<-1+\tfrac{2i}{d+1}$ and negative for $t_i>-1+\tfrac{2i}{d+1}$. This implies that $F$ attains its maximum at $(t_1,\dots, t_n)$ such that $t_i=-1+\tfrac{2i}{d+1}$, for all $i$. Thus,
\begin{equation}
\label{eqn:max}
\max_EF=\sum_{i=1}^d h(-1+\tfrac{2i}{d+1}).
\end{equation}

Now it follows from the versions of \eqref{eqn:specform}, \eqref{eqn:A}, %
\eqref{eqn:max} for $nH$, and the fact that $e-d \le 0$, that for all $n$,
\begin{equation}
\label{eq:c_d(nH)2}
\frac1nc_{d}(nH)\leq \tfrac12\sum_{i=1}^d h(-1+\tfrac{2i}{d+1}).
\end{equation}
To proceed, let $a_1^n\leq \dots\leq a_d^n$ be sequences in $\frac1n\N$ which converge respectively to $-1+\tfrac{2}{d+1}, -1+\tfrac{4}{d+1},\dots, 1-\tfrac{2}{d+1}$. Set $z_i^n$ such that $h'(z_i^n)=a_i^n$. Then,
\[F(z_1^n,\dots, z_d^n)\stackrel{n\to\infty}\longrightarrow \max_EF=\tfrac12\sum_{i=1}^d h(-1+\tfrac{2i}{d+1}).\]
Moreover, we can construct a lattice path, for $nH$, such that \eqref{eqn:A} holds with the $z_i = z^n_i$, and $e = d$.    
We therefore deduce $\frac1nc_{d}(nH)\geq \tfrac12F(z_1^n,\dots, z_d^n)$ and we conclude in combination with \eqref{eq:c_d(nH)2} that
\[ \text{lim}_{n \to \infty} \frac1nc_{d, -d}(nH) = \tfrac12\sum_{i=1}^d h(-1+\tfrac{2i}{d+1}),\]
as desired.
\end{proof}

\begin{remark}
\begin{enumerate}[(i)]
\item It follows from the previous proposition that we have the following Composition property for monotone twists: For any two monotone twists $\phi, \psi$ we have \begin{equation}
\label{eqn:comp}
\mu_d(\phi \psi )= \mu_d(\phi) + \mu_d(\psi).
\end{equation}
Indeed, it follows from Proposition \ref{prop:zeta-d} that $\zeta_d(H_1 + H_2) = \zeta_d(H_1) + \zeta_d(H_2)$, hence \eqref{eqn:comp}.  
\item  For any monotone twist $\varphi \in \Ham(\S^2, \omega)$,
it can easily be shown that $\mu_d(\varphi) = \lim_{n\to \infty} \frac{c_d(\tilde\varphi^n)}{n}$, i.e.\ the $\limsup$ in \eqref{eq:mu_d} is in fact a limit for such $\varphi$.
\end{enumerate}
\end{remark}

  \subsection{A family of Hamiltonians}
  We will now construct a certain family of Hamiltonian diffeomorphisms which will be used to establish Theorem \ref{theo:QI-ker-Cal}.
Let $U\subset \S^2$ be a proper open set containing the North Pole $p_+$ and let $\iota>0$ be an integer. For all $i\in\N$, we denote 
 $$D_i :=  \{ (z, \theta) \in \S^2 :  1 - \tfrac{2}{d_i} < z \le 1\},$$
 where $d_i=2^{\iota+i+1}.$ 
 
 We choose $\iota$ large enough so that all the $D_i$'s are contained in $U$.  Note that each $D_i$ is an embedded disc and   
 $$\area(D_i) =  \frac{1}{d_i}.$$

 The next lemma states the properties of our family of maps.
  \begin{lemma}\label{lem:computation}
   There exist autonomous Hamiltonians $(H_i)_{i\in \N}$,  such that $H_i$ is supported in $D_i$ and the following properties are satisfied for all $t\geq 0$, $i\in\N$:
  \begin{enumerate}
    \item $ \varphi^t_{H_i}$ is a monotone twist, for all $t\in \R$,
\item $\Cal(\varphi_{H_i}^t)=\frac t2$,  %
    \item $d_H(\varphi^t_{H_i}, \id) \leq 2t + 2$,  
  \item If $j >i$, then $\mu_{d_i}(\varphi_{H_j}^t)=-t\frac{d_i}2$
  \item $\mu_{d_i}(\varphi_{H_i}^t)>-t\frac{d_i}2$ and if $j<i$, then
    $\mu_{d_i}(\varphi_{H_j}^t)\geq-t\frac{5d_i}{16}$.
   \item $\varphi^t_{H_i} \varphi^s_{H_j} =  \varphi^s_{H_j} \varphi^t_{H_i}$ for all $t,s \in \R$.
  \end{enumerate}
\end{lemma}

\begin{proof} 
Consider the functions $f_i : [-1, 1] \rightarrow \R$ defined by  
  \[f_i(z) :=
     \begin{cases}
       0 &   z\in [-1,  1 - \frac{2}{d_i}], \\
      d_i^2\left(z - (1 - \frac{2}{d_i})\right),&   z \in [1 - \frac2{d_i}, 1].
     \end{cases}
   \]
These functions are non smooth but to ensure that our future Hamiltonians are smooth, we approximate them by smooth functions $h_i$ satisfying the following conditions:
   \begin{enumerate}[(i)]
   \item $h_i'(z), h_i''(z) \geq 0$ and $|f_i(z) - h_i(z)| \leq  \frac1{d_i}$ for all $z \in [-1, 1]$, 
   \item  The support of $h_i$ is contained in the interior of $[1 -\frac{2}{d_i}, 1]$.
   \item $\int_{-1}^1h_i(z)\,dz=2$.
   \end{enumerate}

Let $H_i(z,\theta) = \frac12h_i(z)$ and observe that $H_i$ is supported in $D_i$ and that $\varphi_{H_i}^t$ is a monotone twist for any $t>0$ (whence item 1 of the lemma).

  The fact that $\Cal(\varphi_{H_i}^t)=\frac t2$ readily follows from Property (iii).

  To prove the third item, we will need the following lemma, whose proof we postpone to the end of this section. The idea behind this lemma goes back to Sikorav, who implemented it in the case of $\R^{2n}$ in \cite{Sikorav-Pisa}; see also \cite[Chap. 5 - 5.6]{hofer-zehnder}.

\begin{lemma}\label{lemma:divide-N}
Let $H:\S^2 \rightarrow \R$ denote an autonomous Hamiltonian such that the support of $H$ is contained in a disc $D$ with the property that $\area(D) < \frac{1}{N}$.  Then,
$d_H(\varphi^1_H, \id) \le \frac{1}{N} \max(H) + 2$.
\end{lemma}

Since the support of $H_i$ is a disk of area less than $\frac1{d_i}$, this lemma leads to
\[  d_H(\varphi_{H_i}^t,\id)\leq \frac{1}{d_i}\max (tH_i)+2\leq \left(1+\frac1{2d_i^2}\right) t +2 \leq 2t + 2,\]
which implies the third item.

Item 4 is a consequence of the Calabi property from \ref{prop:mu_d_properties}, because $\varphi_{H_j}^t$ is supported in $D_j$ and $\mathrm{Area}(D_j)<\frac1{d_i+1}$ for $j>i$.  The last item of Lemma \ref{lem:computation} is also easy to check. Indeed, since the Hamiltonians $H_i$ are functions of $z$, they all Poisson commute, hence their flow commute.

   There remains to prove item 5. Since 
   \[\zeta_{d_i}(tH_j)=\mu_{d_i}(\varphi_{H_j}^t)+d_i\,\int_{\S^2}tH_j\,\omega=\mu_{d_i}(\varphi_{H_j}^t)+t\frac{d_i}2,\]
   we just need to prove that $\zeta_{d_i}(H_i) > 0$ and $\zeta_{d_i}(H_j)\geq \frac{3d_i}{16}$ when $i > j.$ 
  
As already mentioned, the above conditions (i) and (ii) ensure that $\varphi^t_{ H_j}$ is a monotone twist for all $t>0$. 
  By Proposition \ref{prop:zeta-d}, 
  $$\zeta_{d_i}(H_j) = \tfrac12\sum_{m=1}^{d_i} h_j\left(-1 + \frac{2m}{d_i+1}\right).$$
We can rewrite the above sum as 
\begin{align*}
 \zeta_d( H_j) = \tfrac12\sum_{k=0}^N  h_j\left(-1 + \frac{2(d_i-k)}{d_i + 1}\right), 
\end{align*}
where $N$ is the largest integer such that $-1 + \frac{2(d_i-N)}{d_i + 1}$ is in the support of $h_j$,  that is $-1 + \frac{2(d_i-N)}{d_i + 1} > 1 - \frac{2}{d_j}$.  A simple computation reveals that 
\begin{equation}\label{eq:N+1}
N+1 = \frac{d_i}{d_j}.
\end{equation}

Consider the (non-smooth) function $f_j(z)$.  By the definition of $h_j$, we have $h_j \geq f_j - \frac{1}{d_j}$ and so 
\begin{equation}
\zeta_{d_i}(H_j) \geq \left[\frac12\sum_{k=0}^N f_j\left(-1 + \frac{2(d_i-k)}{d_i + 1} \right) \right] - \frac12\frac{N+1}{d_j}.\label{eq:estimating-zeta}
\end{equation}
The first term on the right hand side in the above equation may be computed explicitly. Indeed,  $f_j(z)$ is linear in $z$ and so the above is just an arithmetic sum.  First, note that $-1 + \frac{2(d_i - k)}{d_i + 1} = 1 - \frac{2}{d_{i}+1}
(k+1)$, and so  $\sum_{k=0}^N f_j\left(-1 + \frac{2(d_i-k)}{d_i + 1} \right) = \sum_{k=0}^N f_j\left(1 - \frac{2}{d_{i}+1}
(k+1) \right).$  Next, one can easily check that $f_j\left( 1 - \frac{2}{d_{i}+1}
(k+1) \right) = 2d_j -\frac{2d_j^2}{d_i +1} (k+1)$.  Thus,
   \begin{align}\label{eq:arithmetic-sum}
 \sum_{k=0}^N f_j&\left( 1 - \frac{2}{d_{i}+1}
(k+1) \right) = \sum_{k=0}^N \left[2d_j -\frac{2d_j^2}{d_i +1} (k+1)\right] \nonumber\\
                 &= 2d_j(N+1)- \frac{2d_j^2}{d_i +1} \frac{(N+1)(N+2)}2\nonumber\\
     &= d_i\left(2- \frac{d_i+d_j}{d_i+1}\right).
   \end{align}

For $i=j$,  (\ref{eq:N+1}), (\ref{eq:estimating-zeta}) and (\ref{eq:arithmetic-sum}) yield 
\[\zeta_{d_i}(H_i)\geq d_i\left(1-\frac{d_i}{d_{i}+1}-\frac1{2d_i^2}\right)>0,\]
as desired, which implies the first part of item 5.

For $i>j$,  (\ref{eq:N+1}), (\ref{eq:estimating-zeta}) and (\ref{eq:arithmetic-sum}) give
\begin{align*}
\zeta_{d_i}(H_i)&\geq  \frac{1}{2} d_i\left(2-\frac{d_i+d_j}{d_{i}+1}-\frac 1{d_j^2}\right)\\ &\geq \frac{1}{2} d_i\left(2-\frac{d_i+d_j}{d_{i}}-\frac18\right)= \frac{1}{2}d_i\left(\frac78-\frac{d_j}{d_i}\right)\geq \frac{3d_i}{16}.
\end{align*}
where we used $\frac{1}{d_j^2}\leq \frac{1}{8}$ 
for the second inequality and $\frac{d_j}{d_i}\leq \frac12$ for the last inequality.
This concludes the proof of item 5.  
\end{proof}

We end this section (and the proof of Lemma \ref{lem:computation}) with the proof of Lemma \ref{lemma:divide-N}.

\begin{proof}[Proof of Lemma \ref{lemma:divide-N}] 
 Since $\area(D) < \frac{1}{N}$, 
 by Lemma \ref{lem:transport_energy}, we can find  $\psi_1, \ldots, \psi_N \in \Ham(\S^2, \omega)$ such that 
\begin{enumerate}
\item $\psi_i(D) \cap \psi_j(D) = \emptyset$,
 \item $d_H(\psi_i, \id) \le \frac{1}{N}$.
\end{enumerate}

Define the Hamiltonians $F_i = \frac{1}{N} H \circ \psi_i^{-1}$ and note that $\varphi^1_{F_i} = \psi_i \varphi^{\frac{1}{N}}_{H} \psi_i^{-1}$.  Let  $F = \sum_{i=1}^N F_i$.  Observe that,  since $F_i$ is supported in $\psi_i(D)$, the $F_i$'s have disjoint supports, and so $\max(F) = \max(F_i) = \frac{1}{N} \max(H)$.  Hence, 
$$d_H(\varphi^1_F, \id) \leq \frac{1}{N} \max(H),$$
where the last line follows from the definition of the Hofer norm.
  Therefore, to prove the claim it is sufficient to show that $d_H(\varphi^1_H, \varphi^1_F) \leq 2$. This can be proved using Identities (\ref{eq:Hofer-identity1}) and (\ref{eq:Hofer-identity2}) as follows:
\begin{align*}
d_H(&\varphi^1_H, \varphi^1_F)  =  d_H\left(\prod_{i=1}^N \varphi^{\frac{1}{N}}_{H}, \prod_{i=1}^N \psi_i \varphi^{\frac{1}{N}}_{H} \psi_i^{-1}\right)\\
& \leq \sum_{i=1}^N d_H( \varphi^{\frac{1}{N}}_{H} ,\psi_i \varphi^{\frac{1}{N}}_{H} \psi_i^{-1}) \leq \sum_{i=1}^N 2 \, d_H(\psi_i, \id) \leq 2.
\end{align*} 
\end{proof}

 \subsection{Quasi-flats in the kernel of Calabi}\label{sec:proof_embedding}
 We are now ready to present a proof of Theorem \ref{theo:QI-ker-Cal}.  First note that without loss of generality 
 we may assume that the open set $U$ contains the North pole $p_+$. This allows us to use the constructions of the preceding section.
 
 We begin by proving the first part of the theorem, regarding the quasi-flat rank.
 
 \begin{proof}[Proof of Theorem \ref{theo:QI-ker-Cal}(a)]

Let $\R^n_+ := \{(t_1, \ldots, t_n) : t_i \ge 0\}$.  We equip $\R^n_+$ with the distance induced by the sup norm, that is we define the distance between $(t_1, \ldots, t_n), (s_1, \ldots, s_n) \in \R^n_+$ to be   $$\Vert (t_1, \ldots, t_n) - (s_1, \ldots, s_n) \Vert_{\infty} = \max \{\vert t_i - s_i \vert :  i= 1, \ldots, n \}.$$
The mapping 
\begin{align*}
\Phi : \R^n_+ &\rightarrow \Ham(\S^2, \omega) \\
 (t_1, \ldots, t_n) &\mapsto  \varphi^{t_1}_{H_{1}} \circ \ldots \circ \varphi^{t_n}_{H_{n}}\circ\varphi_{H_{n+1}}^{-(t_1+\dots+t_n)},
\end{align*}
takes values $\Ham_U(\S^2,\omega)$. Moreover, since the $H_i$ all have the same integral over $\S^2$, the mapping $\Phi$ takes values in the kernel of the Calabi homomorphism.

We will show that there exists an invertible $ n\times n$ matrix $A$ such that $\Phi$ satisfies the following inequality.
\begin{equation}\label{eq:quasi_isom-matrix}
      \Vert A({\bf t} - {\bf s}) \Vert_\infty \leq d_H(\Phi({\bf t}), \Phi( {\bf s} )) \leq 2n\,  (2\Vert {\bf t} - {\bf s} \Vert_\infty +1+\tfrac1n) ,
   \end{equation}
where ${\bf t}$, ${\bf s}$ stand for $(t_1, \ldots, t_n)$ and $(s_1, \ldots, s_n)$, respectively. As a consequence, 
\begin{equation*}\label{eq:quasi_isom}
     \frac1{\|A^{-1}\|_{\mathrm{op}}} \Vert {\bf t} - {\bf s} \Vert_\infty \leq d_H(\Phi({\bf t}), \Phi( {\bf s} ) ) \leq 4n  \Vert {\bf t} - {\bf s} \Vert_\infty +2n+2 ,
   \end{equation*}
where $\|A^{-1}\|_{\mathrm{op}}$ is the operator norm of $A^{-1}$, as a linear map of the normed space $(\R^n, \|\cdot\|_\infty)$.

  The above clearly implies that $\Phi$ is a quasi-isometric embedding of $(\R^n_+, \Vert \cdot \Vert_\infty )$ into $(\Ham(\S^2, \omega),d_H)$.  This establishes Theorem \ref{theo:QI-ker-Cal} because $(\R^{n}, \Vert \cdot \Vert_\infty )$ quasi-isometrically embeds into  $(\R^{2n}_+, \Vert \cdot \Vert_\infty )$; an explicit formula for such a quasi-isometric embedding is given by  
\begin{align*}
{\bf L} :  \R^n  & \rightarrow \R^{2n}_+\\
(x_1, \ldots, x_n) & \mapsto (L(x_1), \ldots, L(x_n)),
\end{align*}
where $L : \R \rightarrow \R^2_+$ is defined as

  \[   L(x) :=
     \begin{cases}
       (0, -x),   &   x \le 0, \\
       (x, 0),  &  x \ge 0.
     \end{cases}
   \]
 For a proof of the fact that $\bf L$ is a quasi-isometric embedding see \cite[Lem.\  8.12]{Stoj-Zhang}.

 We now turn our attention to \eqref{eq:quasi_isom-matrix}, beginning with the following proof of the inequality on its right-hand side:
\begin{align*}
 d_H &\left(\Phi({\bf t}), \Phi( {\bf s}) \right)
  =  \Vert \Phi({\bf t}) \circ \Phi( {\bf s})^{-1}\Vert =  \Vert \varphi^{t_1 -s_1}_{H_1} \ldots \varphi^{t_n -s_n}_{H_n}\varphi_{H_{n+1}}^{-\sum_i(t_i-s_i)}\Vert \\
     &\leq \Vert \varphi_{H_{n+1}}^{-\sum_i(t_i-s_i)}\Vert + \sum_{i=1}^n \Vert \varphi^{t_i -s_i}_{H_i} \Vert  \leq  \sum_{i=1}^n \left( 4|t_i - s_i| + 2 \right)+2\\
  &\leq 2n \left( 2\Vert {\bf t} - {\bf s} \Vert_\infty +1 + \frac{1}{n}\right).
\end{align*}

Above, the second equality on the first line is a consequence of the last item in Lemma \ref{lem:computation}, the first inequality on the second line  follows from triangle inequality and the second inequality in the second line is a consequence of Lemma \ref{lem:computation}, item 3.   This proves the right hand side of \eqref{eq:quasi_isom-matrix}.

It remains to prove the left hand side.  Let us consider the following two families of monotone twists.
\[\alpha({\bf t})=\varphi^{t_1}_{H_{1}}\ldots \varphi^{t_n}_{H_{n}},\quad \beta({\bf t})=\varphi_{H_{n+1}}^{(t_1+\dots+t_n)}.\]
By definition $\Phi({\bf t})=\alpha({\bf t})\beta({\bf t})^{-1}$ and since monotone twists commute 
\begin{align*}
  d_H(\Phi({\bf t}), \Phi( {\bf s} ))&=\|\beta({\bf t})\alpha({\bf t})^{-1}\alpha({\bf s})\beta({\bf s})^{-1}\|
  =\|(\alpha({\bf t})\beta({\bf s}))^{-1}\alpha({\bf s})\beta({\bf t})\|\\&= d_H(\alpha({\bf t})\beta({\bf s}),\alpha({\bf s})\beta({\bf t})).
\end{align*}
Now, combining the previous equality with the second item of Proposition \ref{prop:mu_d_properties} gives 
\begin{equation}\label{eq:lower_bound1}
\max_{i=1,\dots,n} \left\vert \; \frac{\mu_{d_i}(\alpha({\bf t})\beta({\bf s}))}{d_i} - \frac{\mu_{d_i}(\alpha({\bf s})\beta({\bf t}))}{d_i} \; \right\vert \le  d_H(\Phi({\bf t}), \Phi( {\bf s} ) ).
\end{equation}
By the fourth item 
of Proposition \ref{prop:mu_d_properties} and \eqref{eqn:comp}, we can write,
$$\mu_{d}(\alpha({\bf t})\beta({\bf s})) - \mu_{d}(\alpha({\bf s})\beta({\bf t}))  = \sum_{j=1}^n  \left(\mu_{d}(\varphi^{1}_{H_j})-\mu_d(\varphi^{1}_{H_{n+1}})\right)(t_j - s_j),$$
for any $d$.
It follows from the above that the left hand side in \eqref{eq:lower_bound1} coincides with the quantity $$\Vert A ({\bf t} - {\bf s})\Vert_{\infty},$$ where $A$ is the matrix whose $ij$ entry (for $1\leq i,j\leq n$) is
$$A_{ij} = \frac{\mu_{d_i}(\varphi^1_{H_j})-\mu_{d_i}(\varphi^1_{H_{n+1}})}{d_i}.$$

The fourth item in Lemma \ref{lem:computation} tells us that $\mu_{d_i}(\varphi^1_{H_{n+1}}) = \mu_{d_i}(\varphi^1_{H_{j}})=-\frac{d_i}2$, for $j > i$.  It follows that  $A_{ij}=0$ for $j>i$, i.e. the matrix $A$ is lower triangular.  From the fifth item of the same lemma, we deduce that the diagonal entries of $A$ are non-zero.  Hence, $A$ is invertible which proves (\ref{eq:quasi_isom-matrix}). We have completed the proof of Theorem \ref{theo:QI-ker-Cal}(a).
\end{proof}

\subsection{The kernel of Calabi is not coarsely proper}
In this section, we prove the remainder of Theorem \ref{theo:QI-ker-Cal}, i.e. that the kernel of the Calabi Homomorphism defined on $\Ham_U(\S^2,\omega)$ is not coarsely proper. Recall from the introduction that a metric space $(X,d)$ is said to be coarsely proper if there exists $R_0 > 0$ such that every bounded subset of $(X,d)$ can be covered by finitely many balls of radius $R_0$. %

\begin{proof}[Proof of Theorem~\ref{theo:QI-ker-Cal}(b)] For any fixed $r >0$ consider the set  $$X_r:= \{ \varphi_{H_1}^{-r}\varphi^r_{H_{2i}} : i\geq 1\},$$ 
where the $H_{i}$ are the Hamiltonians provided by Lemma \ref{lem:computation}.  
 
 \begin{claim}\label{cl:coarsely_proper}
 The set $X_r$ is $\frac{3r}{16}$ 
 separated, i.e. for $i \neq j$ we have $$ \frac{3r}{16} \le d_H(\varphi_{H_1}^{-r}\varphi^r_{H_{2i}}, \varphi_{H_1}^{-r}\varphi^r_{H_2j}) .$$
 \end{claim}
 
 The above claim implies that the set $X_r$, which is bounded by Lemma \ref{lem:computation}, 
 cannot be covered by finitely many balls of radius $\frac{r}{16}$. 
 Since this holds for every value of $r$, and since $X_r$ is included in $\ker(\Cal)$, we conclude that $\ker(\Cal)$ is not coarsely proper, hence the Theorem.

 It remains to prove Claim \ref{cl:coarsely_proper}.  
 \begin{proof}[Proof of Claim \ref{cl:coarsely_proper}]
   Suppose that $ i<j$, pick $k \in \N$ such that $2i < k <2j$ and consider $d_k$ as in Lemma \ref{lem:computation}.
   By the Hofer continuity property of $\mu_{d_k}$, from Proposition \ref{prop:mu_d_properties}, we have 
  $$\frac{1}{d_k}| \mu_{d_k}(\varphi^r_{H_{2i}}) - \mu_{d_k}(\varphi^r_{H_{2j}})| \leq d_H(\varphi^r_{H_{2i}}, \varphi^r_{H_{2j}})= d_H(\varphi_{H_1}^{-r}\varphi^r_{H_{2i}}, \varphi_{H_1}^{-r}\varphi^r_{H_{2j}}).$$
  By Lemma \ref{lem:computation}, we have  $ \mu_{d_k}(\varphi^r_{H_{2j}})=-r\frac{d_k}2$ and 
  $ \mu_{d_k}(\varphi^r_{H_{2i}}) \geq -r\frac{5d_k}{16}$. 
 Thus, 
 $\frac{1}{d_k}| \mu_{d_k}(\varphi^r_{H_{2i}}) - \mu_{d_k}(\varphi^r_{H_{2j}})|\geq \frac {3r}{16}$ 
 which completes the proof.
\end{proof}
 
 We have now proved Theorem~\ref{theo:QI-ker-Cal}(b).
  \end{proof}

 \section{Non-simplicity of $\Homeo_0(\S^2,\omega)$}
 \label{sec:nonsimp}
 
 We conclude by proving Theorem \ref{theo:non-simplicity}.
 
 \subsection{Outline of the argument}
 \label{sec:outline}
 
 To prove non-simplicity of $\Homeo_0(\S^2, \omega)$, we explicitly construct a proper normal subgroup which we call the group of {\bf finite energy homeomorphisms} and denote by $\FHomeo(\S^2, \omega)$.  We introduced these homeomorphisms in \cite{CGHS} where we proved that they form a proper normal subgroup of the compactly supported area-preserving homeomorphisms of the disc.  Here, we will give a slight variant of the definition in \cite{CGHS} which is more natural from the point of view of Hofer's geometry.

\begin{definition}\label{def:FHomeo}
We say $\varphi \in \Homeo_0(\S^2, \omega)$ is a finite-energy homeomorphism if there exists a sequence of Hamiltonian diffeomorphisms $\{\varphi_i\}_{i\in \N}$ which is bounded with respect to Hofer's distance and which converges uniformly to $\varphi$.  We denote by $\FHomeo(\S^2, \omega)$ the set of all finite-energy homeomorphisms.
\end{definition}

Theorem \ref{theo:non-simplicity} follows immediately from the following result, which will occupy the remainder of this section.

\begin{theo}\label{theo:FHomeo_proper}
$\FHomeo(\S^2, \omega)$ is a proper normal subgroup of $\Homeo_0(\S^2, \omega)$.
\end{theo}

We prove the above using arguments similar to those  given in \cite{CGHS}.  Here is a brief outline.  As we shall see, it is not hard to show that $\FHomeo(\S^2, \omega)$ forms a normal subgroup of $\Homeo_0(\S^2, \omega)$; the main challenge is proving that it is proper.

To do this, we use the invariant $\eta_d: \Ham(\S^2, \omega)\rightarrow \R$.  We showed above that this  is continuous with respect to the $C^0$ topology on $\Ham(\S^2, \omega)$ and, moreover, it extends continuously to $\Homeo_0(\S^2, \omega)$; see Proposition \ref{prop:eta_d_properties}.  A straightforward argument
shows that for any $\varphi \in \FHomeo(\S^2, \omega)$ there exists a constant $C$, depending on $\varphi$, such that for all (even) $d$ we have
\begin{equation}\label{eq:eta_d_bdd}
\frac{\eta_d(\varphi)}{d} \leq C.
\end{equation}

 We will then prove Theorem \ref{theo:FHomeo_proper} by showing that certain so-called {\bf infinite twist} homeomorphisms $\psi \in \Homeo_0(\S^2, \omega)$ satisfy the following;
\begin{equation}\label{eq:sup_linear_growth}
\lim_{d \to \infty} \frac{\eta_d(\psi)}{d} = \infty.
\end{equation}
This violates \eqref{eq:eta_d_bdd}.    This last step requires estimating $\eta_d(\psi)$ for which we rely on the combinatorial model from Section \ref{sec:comb}.

\medskip

We end this section by highlighting the differences between our proof, in this article, of non-simplicity of $\Homeo_0(\S^2, \omega)$ and the proof of non-simplicity of $\Homeo_c(\D^2, \omega) $ given in \cite{CGHS}.  In both articles we use PFH spectral invariants $c_d : C^{\infty}(\S^1 \times \S^2) \to \R$.  Given an arbitrary Hamiltonian $H$, the value of $c_d(H)$ depends on $H$ and so $c_d$ does not yield a well-defined invariant of Hamiltonian diffeomorphisms.  However, in \cite{CGHS} we overcome this problem by restricting the domain of $c_d$ to a certain class of Hamiltonians which is suitable for the purposes of that article; see \cite[Sec.\ 3.4]{CGHS}.  In the current article, we do not have the possibility of restricting the domain of $c_d$.  Instead, we work with $\eta_d$ which is well-defined for all Hamiltonian diffeomorphisms of the sphere as proved in Section \ref{sec:eta_d}.

Another difference between the two proofs is the manner in which we show properness of $\FHomeo$. In both articles this is achieved by exhibiting area-preserving homeomorphisms $\psi$ satisfying Equation \eqref{eq:sup_linear_growth}. The proof of this given in \cite{CGHS} involves verifying for certain smooth twist maps a conjecture of Hutchings, concerning recovering the Calabi invariant from the asymptotics of PFH spectral invariants, whereas our proof here, which is shorter, uses the forthcoming Claim~\ref{clm:nice}.  The proof of 
this claim, however, relies on the combinatorial model for PFH developed in \cite[Sec.\ 5]{CGHS}.  We should remark that part of the motivation for the somewhat longer argument in \cite{CGHS} was that Hutchings' conjecture is of independent interest, hence useful to verify for twist maps.

\subsection{An infinite twist is not a finite energy homeomorphism}\label{sec:an-infinite-twist}

We now carry out the above outline.  We begin by describing the infinite twist homeomorphisms $\psi$.

Denote by $p_+ \in \S^2$ the North Pole of the sphere, i.e.\ the point whose $z$-coordinate is $1$, in the cylindrical coordinate system introduced in Section \ref{sec:prelim_symp}.  We say a function $F:\S^2 \setminus \{p_+\} \rightarrow\R$ is an {\bf infinite twist Hamiltonian} if it is of the form  
\begin{equation}\label{eq:inf_twist_Hamiltonian}
F(z,\theta) = \frac 12 f(z),
\end{equation}
where $f :[-1, 1) \rightarrow \R$ is a smooth function such that $f' \ge 0, f'' \ge 0$ and
\begin{equation}\label{eq:infinite_asymptotics}
\lim_{d \to \infty}\frac1d \,f\left(1-\frac{2}{d+1}\right) = \infty.
\end{equation}
Observe that $F$ defines a smooth Hamiltonian on $\S^2 \setminus \{p_+\}$ whose flow is given by 
$$\varphi^t_F(\theta, z) = (\theta + 2 \pi t f'(z), z).$$
  We extend the flow $\varphi^t_F$ to $\S^2$ by defining $\varphi^t_F(p_+) = p_+$; this yields an area-preserving flow on $\S^2$ which is non-smooth at the point $p_+$.   We say $\psi \in \Homeo_0(\S^2, \omega)$ is an {\bf infinite twist homeomorphism} if it is of the form
\begin{equation}
\psi := \varphi^1_F
\end{equation}
for some $F$.    We will call $\psi$ an {\bf adapted} infinite twist if the corresponding $f$ satisfies the following technical hypothesis:
\[ f'\left(1 - \frac{2}{d+1}\right) \in (d+1)\N,\]
for all $d \ge 2$. 

We can now give the promised proof of the remaining theorem.

\begin{proof}[Proof of Theorem \ref{theo:FHomeo_proper}]

We begin by noting that the argument in \cite[Prop. 2.1]{CGHS}, repeated verbatim, shows that $\FHomeo(\S^2, \omega)$ forms a normal subgroup of $\Homeo_0(\S^2, \omega)$. It remains to show that it is proper.

\medskip
\emph{Step 1.  Linear growth in $\FHomeo$.}  

We first show that for any $\varphi \in \FHomeo(\S^2, \omega)$ the linear growth condition \eqref{eq:eta_d_bdd} holds.  This is an immediate consequence of the properties in Proposition \ref{prop:eta_d_properties}.  Indeed, let $\varphi \in \FHomeo(\S^2, \omega)$ and choose a sequence $\varphi_i \xrightarrow{C^0} \varphi$ that is uniformly bounded with respect to Hofer's distance.  Since the $\varphi_i$ are bounded and $\eta_d(\id) = 0$, the Hofer continuity property ensures a bound of the form $ \eta_d(\varphi_i) \le d \cdot C$ for some uniform constant $C$; then, by $C^0$ continuity, the same bound holds for $\varphi$.

\medskip
\emph{Step 2.  Superlinear growth of some infinite twists.}

It remains to prove that $\FHomeo$ is proper.  The structure of the remainder of our argument will now be to prove
properness,
assuming the technical Claim \ref{clm:whatever} below which makes use of the adapted condition, and then prove the Claim.  
From now until the end of the paper, we therefore assume that $F$ is an adapted infinite twist Hamiltonian whose support is contained in the interior of the disc $\{(z, \theta) : z \ge  \frac 78\}$ which is of area $\frac{1}{16}$.  Imposing this assumption enables us to apply the following promised technical claim.  Recall below that the $\eta_d$ are defined only for even $d$.  

\begin{claim}
\label{clm:whatever}
Fix $d \ge 4,$ define $z_0 := 1 - \frac{2}{d+1}$. Let $H$ be a smooth 
monotone twist Hamiltonian supported in a disc of area at most $1/12$.  Assume that $p := h'(z_0) \in (d+1)\N$. 
Then
\[ \eta_d(\varphi_H^1) \ge H(z_0) - \frac{d}{6}.\]
\end{claim}

We defer the proof for the moment.  Assuming it, we can produce super linear growth of the $\eta_d$ as follows.  

\begin{claim}
\label{clm:nice}
$\eta_d(\varphi^1_F)  \ge \frac{1}{2} f(1 - \frac{2}{d+1}) - \frac d6$, for $d \ge 4$.
\end{claim}
\begin{proof}[Proof of Claim \ref{clm:nice}]
For every $i \in \N$, let $F_i : \S^2 \rightarrow \R$ be a sequence of smooth Hamiltonians of the form

$$F_i(z, \theta) = \frac 12 f_i(z),$$
where $f_i :[-1, 1] \rightarrow \R$ is a smooth function such that $f_i' \ge 0, f_i'' \ge 0$ and 
$f_i(z) = f(z)$ for $z\in [-1, 1 - \frac 1i]$. 

Observe that $\varphi^1_{F_i} \xrightarrow{C^0} \varphi^1_F$, because $(\varphi^1_{F_i})^{-1} \circ \varphi^1_F$ is supported in the disc $\{(z, \theta) : z \ge 1 - \frac{1}{i}\}$.  Hence, by the $C^0$ continuity of $\eta_d$ established in Proposition \ref{prop:eta_d_properties}, we have
$$ \eta_d(\varphi^1_F) = \lim_{ i \to \infty} \eta_d(\varphi^1_{F_i}). %
$$
for every $d$.
Applying Claim \ref{clm:whatever} to $F_i$, for $i$ sufficiently large with respect to $d$, yields $$\eta_d(\varphi^1_{F_i}) \ge F_i\left(1 - \tfrac{2}{d+1}\right)  - \tfrac{d}{6}  = \tfrac 12 f_i \left(1 -\tfrac{2}{d+1}\right)  -\tfrac{d}{6}=  \tfrac{1}{2} f\left(1 -\tfrac{2}{d+1}\right) -\tfrac d6,$$
for $d > 3$. Hence, the claim.  
\end{proof}
It follows immediately from the previous claim that $\psi:= \varphi^1_F$ satisfies Equation \eqref{eq:sup_linear_growth} which, as explained in Step 1, implies that an adapted infinite twist $\varphi^1_F$ is not a finite-energy homeomorphism.

 To complete the proof of Theorem \ref{theo:FHomeo_proper}, it therefore remains to prove Claim \ref{clm:whatever}.

\begin{proof}[Proof of Claim \ref{clm:whatever}] Recall, from Equation \eqref{eqn:eta_d2}, that $\eta_d(\varphi) = c_{d}(H) -\frac{d}{2}c_{2}(H)$.  Hence, to prove the Claim, it is sufficient to show that the following two inequalities hold:
\begin{equation}\label{eq:bound_c_2}
 c_{2}(H) \leq \frac{1}{3}.
\end{equation}
\begin{equation}\label{eq:bound_C_d}
c_{d}(H) \ge H\left(1- \frac{2}{d+1}\right).
\end{equation}

To prove \eqref{eq:bound_c_2}, we invoke the Support-control inequality of Proposition \ref{prop:more}, which gives $c_2(H) \le 2 \cdot 2 \cdot \frac{1}{12}$, since the area of the support of $H$ is bounded by $\frac{1}{12}$.

Next, we prove \eqref{eq:bound_C_d}.  By the Continuity property of $c_d$
from Theorem \ref{thm:PFHspec_initial_properties}, we may perform a small perturbation of $h$, near $z=1$, and assume that $h'(1) \in \N$, in other words that our twist is nice.  
This allows us to apply Theorem 6.1 of \cite{CGHS} whose statement we recalled in Section \ref{sec:comb}.

Recall the notation $z_0 = 1 - \frac{2}{d+1}$.  By Theorem 6.1 of \cite{CGHS} 
we have 
$$c_{d}(H) \ge \mathcal{A}(P),$$
for any degree $d$ lattice path $P$ of combinatorial index $I(P)=-d$; see Section \ref{sec:comb}.

Recall the notation $p := h'(z_0)$, which is by assumption an integer. 
By assumption, there exists an integer $a>0$ such that $p=a(d+1)$.
Take $P$ to be the lattice path obtained by joining the lattice points $(0,-a)$, $(d-1, -a)$ and $(d,p-a)$. This is a concave lattice path made of two edges. It satisfies
\begin{align*}
  \mathcal A(P) &=\frac{p}{2} (1-z_0) + \frac12 h(z_0)-a,\\
  I(P)&=2j(P)-d= 2((p - a)-da)-d=-d.
\end{align*}
Hence, 
\begin{align*}
  c_{d}(H) &\ge  \frac{p}{2} (1-z_0) + \frac12 h(z_0) - a =  \frac{p}{2} \left( 1 -z_0 - \frac{2}{d+1}\right) + \frac12 h(z_0) \\
  &=\frac12 h(z_0) = H\left(1 - \frac{2}{d+1}\right).
\end{align*}
\end{proof}
We have completed the proof of Theorem \ref{theo:FHomeo_proper}.
\end{proof}

\begin{remark}\label{rem:quotient}
The infinite twist Hamiltonian $F$, introduced above, generates a 1-parameter subgroup $\varphi^t_F$ of $\Homeo_0(\S^2, \omega)$.  It follows immediately from the above proof that $\varphi^t_F \notin \FHomeo(\S^2, \omega)$ for $t \neq 0$.  This yields an injective group homomorphism from the real line $\R$ into the quotient $\Homeo_0(\S^2, \omega)/\FHomeo(\S^2, \omega)$.  One can show that this injection is not a surjection; see \cite{Pol-Shel2}.  However, we have not been able to determine whether $\Homeo_0(\S^2, \omega)/\FHomeo(\S^2, \omega)$ is isomorphic to $\R$ as an abelian group.
\end{remark}

\begin{remark}\label{rem:genus_zero_surfaces}
The group of finite energy homeomorphisms $\FHomeo(\Sigma, \omega)$ can be defined on any surface $\Sigma$; it  forms a normal subgroup of $\overline{\Ham}_c(\Sigma, \omega)$, the group of Hamiltonian homeomorphisms of $(\Sigma, \omega)$ which coincide with the identity near the boundary, if $\partial \Sigma \neq \emptyset$.     Recall that a Hamiltonian homeomorphism is a  homeomorphism which can be written as the $C^0$ limit of Hamiltonian diffeomorphisms.  It is well-known that $\overline{\Ham}_c(\Sigma, \omega)$ coincides with the kernel of the mass-flow homomorphism of Fathi \cite{fathi}.

Suppose now that $(\Sigma, \omega)$ is any compact surface of genus $0$, with boundary,  and view it as  embedded into $(\S^2, \omega)$.  There is an inclusion
$\FHomeo(\Sigma, \omega) \subset \FHomeo(\S^2, \omega)$.  The infinite twist $\psi$ can be placed on $(\Sigma, \omega)$ and the fact that it is not a finite-energy homeomorphism of the sphere implies that $\psi \notin \FHomeo(\Sigma, \omega)$.  We conclude that $\overline{\Ham}_c(\Sigma, \omega)$ is not simple.  Moreover, as in the case of Corollary \ref{cor:non-perfect}, one can  conclude that $\overline{\Ham}_c(\Sigma, \omega)$ is not perfect either.  This answers a question of Fathi \cite[Appendix A.6]{fathi}, concerning the simplicity of the kernel of the mass-flow homomorphism, for compact genus-zero surfaces. 

We remark that the infinite twist and $\FHomeo$ can be defined on any symplectic manifold.  However, our methods for proving properness of finite-energy homeomorphisms are currently applicable to dimension two only.
\end{remark}

\subsection{Proof of Corollary \ref{non-perfectness-R2}}\label{sec:non-perfectness-R2}
We now give the proof of Corollary~\ref{non-perfectness-R2}.  

\begin{proof} We may assume without loss of generality that $\int_{\R^2}\Omega=1$. As alluded to in the introduction, by a version of Moser's argument for non-compact manifolds, due to Greene and Shiohama \cite{Greene-Shiohama}, there exists a smooth diffeomorphism $\psi:\R^2\to \S^2\setminus\{p\}$ such that $\psi^*\omega=\Omega$. Here $p$ denotes the North pole in $\S^2$. This gives rise to an injective group homomorphism $\Psi:\Diff(\R^2,\Omega)\to \Homeo_0(\S^2,\omega)$, defined for any $h\in \Diff(\R^2,\Omega)$ by $\Psi(h)(x):=\psi h \psi\inv(x)$ for $x\neq p$ and  $\Psi(h)(p)=p$. The image of $\Psi$ is the set of elements of $\Homeo_0(\S^2,\omega)$ that fix $p$ and are smooth in the complement of $p$.

In particular, the image of $\Psi$  contains an adapted infinite twist homeomorphism $\tau$, which we showed above is not in $\text{FHomeo}(\S^2,\omega).$  By the Epstein-Higman argument cited in the introduction, the commutator subgroup of $\text{Homeo}_0$ is contained in any non-trivial normal subgroup.  In particular, $\tau$ is not a product of commutators. This implies that $\Psi^{-1}(\tau)$ is not a product of commutators in $\Diff(\R^2,\Omega)$, which is therefore not perfect. 
\end{proof}

\begin{remark} The above argument similarly shows that $\Homeo(\R^2, \Omega)$, the group of area-presrving homeomorphisms of the plane, is not perfect if $\Omega$ has finite total area.  This holds more generally if $\Omega$ is only assumed to be a good measure and not necessarily a smooth area form; being good means that $\Omega$ is non-atomic and positive on non-empty open sets. In this case, one can repeat the above argument, using the classical Oxtoby-Ulam theorem \cite{Oxtoby-Ulam} instead of \cite{Greene-Shiohama}.  
\end{remark}

 \subsection{Remarks on Hofer's geometry}
 
We close by briefly discussing the large scale geometry of $\FHomeo$.

It is possible to define Hofer's distance for area-preserving homeomorphisms as follows.  Given $\varphi \in \FHomeo(\S^2, \omega)$, we define its Hofer distance from the identity by
\begin{equation}\label{eq:Hofer_distance_FHomeo}
\tilde{d}_H(\varphi, \id) := \liminf d_H(\varphi_i, \id), 
\end{equation}
where the infimum is taken over all sequences $\{\varphi_i\} \subset \Ham(\S^2, \omega)$ which converge uniformly to $\varphi$.  Define $\tilde{d}_H(\varphi,\psi) := \tilde{d}_H(\varphi^{-1} \psi,\id)$.  

We leave it to the reader to check that this defines a bi-invariant distance on $\FHomeo(\S^2, \omega)$.  

It is a natural question to try to better understand this space.  For example, one could ask if  $\FHomeo$ has infinite quasi-flat rank.  We strongly suspect that the answer is, in fact, positive as our tools are robust with respect to the $C^0$ topology and so one can adapt the proof of Theorem \ref{theo:QI-ker-Cal} to prove that $\FHomeo$  does have infinite quasi-flat rank.  Similarly, it can be shown that $\FHomeo$ is not coarsely proper.

One could define $\tilde d_H(\varphi, \id)$, via \eqref{eq:Hofer_distance_FHomeo}, for arbitrary $\varphi \in \Homeo_0(\S^2, \omega)$.  

If $\varphi$ is not a finite energy homeomorphism, i.e.\ if $\varphi \notin \FHomeo(\S^2, \omega)$, then we get $$\tilde d_H(\varphi, \id) = \infty.$$
Hence, we may view homeomorphisms which are not finite-energy as those which are infinitely far from diffeomorphisms, in Hofer's distance. This is the point of view expressed in Le Roux's article \cite[Question 1]{LeRoux-6Questions}. Theorem \ref{theo:FHomeo_proper} tells us that such homeomorphisms do exist.

A question which arises immediately as a consequence of our definition of $\tilde d_H$ is whether $\tilde d_H(\varphi, \psi)$ coincides with the usual Hofer distance $ d_H(\varphi, \psi)$ when $\varphi, \psi \in \Ham(\S^2, \omega)$.  We do not know the answer to this question.  Note that this is equivalent to asking if the (usual) Hofer distance is lower semi-continuous with respect to the $C^0$ topology; this was posed as an open question by Le Roux in  \cite{LeRoux-6Questions}.  

 \bibliographystyle{alpha}
 \bibliography{Kap-Pol}

{\small

\medskip
\noindent Dan Cristofaro-Gardiner\\
Mathematics Department \\ 
University of California, Santa Cruz \\
1156 High Street, Santa Cruz, California, USA\\
School of Mathematics \\
Institute for Advanced Study\\
1 Einstein Drive, Princeton, NJ, USA \\
{\it e-mail}: dcristof@ucsc.edu
\medskip

\medskip
\noindent Vincent Humili\`ere \\
\noindent Sorbonne Universit\'e and Universit\'e de Paris, CNRS, IMJ-PRG, F-75006 Paris, France\\
\noindent \& Institut Universitaire de France.\\
{\it e-mail:} vincent.humiliere@imj-prg.fr
\medskip

\medskip
 \noindent Sobhan Seyfaddini\\
\noindent Sorbonne Universit\'e and Universit\'e de Paris, CNRS, IMJ-PRG, F-75006 Paris, France.\\
 {\it e-mail:} sobhan.seyfaddini@imj-prg.fr

}

\end{document}